%% file: QuadraticWave.tex
\documentclass[preprint,12pt]{elsarticle}

\journal{}
\usepackage{amsmath,amssymb,amsthm,a4wide,color,graphicx,comment}
\usepackage{xstring,xparse, mathtools}
\usepackage{mathrsfs}
\usepackage[english]{babel}
\usepackage{dsfont}
\usepackage{scalerel}
\usepackage{hyperref}
\usepackage{verbatim}
\usepackage{autonum}
\usepackage{booktabs}
\newtheorem{theorem}{Theorem}[section]
\newtheorem{lemma}{Lemma}[section]

\newtheorem{remark}{Remark}[section]
	\newcommand{\spacevar}{\fat{x}}
	\newcommand{\timevar}{t}
	\newcommand{\entr}{s}
	\newcommand{\temp}{\theta}
	\newcommand{\dens}{\rho}
	\newcommand{\velo}{\fat{v}}
	\newcommand{\welo}{\fat{w}}
	\newcommand{\pres}{p}
	\newcommand{\qres}{q}
	\newcommand{\forcep}{h}
	\newcommand{\forcev}{\fat{g}}
	\newcommand{\forces}{\ell}
	\newcommand{\newx}{\fat{x}'}
	\newcommand{\newt}{\timevar'}
	\newcommand{\newr}{\dens'}
	\newcommand{\newv}{\fat{v'}}
	\newcommand{\newp}{\pres'}
	\newcommand{\newent}{\entr'}
	\newcommand{\newthe}{\temp'}

	\newcommand{\bvis}{\mu_B}
	\newcommand{\svis}{\mu_S}
	\newcommand{\etap}{\mu}
	\newcommand{\etav}{\eta}
	\newcommand{\thecon}{\kappa}
	
	\newcommand{\ident}{\mathcal{I}}
	\newcommand{\linearspace}{\mathscr{L}}
	\newcommand{\dime}{d}
	\newcommand{\transpose}{\top}
	\newcommand{\R}{\mathbb{R}}
	\newcommand{\C}{\mathbb{C}}
	\newcommand{\N}{\mathbb{N}}
	
	\newcommand{\T}{\mathcal{T}}
	\newcommand{\F}{\mathcal{F}}
	\newcommand{\e}{\text{e}}
	\newcommand{\Lcal}{\mathcal{L}}
	\newcommand{\Ncal}{\mathcal{N}}
	\newcommand{\B}{\mathcal{B}}
	\newcommand{\A}{\mathcal{A}}
	\newcommand{\BBB}{\mathbb{B}}
	\newcommand{\AAA}{\mathbb{A}}
	\newcommand{\Fprime}[2]{\F'_{#1}[#2]}
	\newcommand{\eps}{\varepsilon}
	\newcommand{\pro}{\mathbf{P}}
	\newcommand{\embedded}{\hookrightarrow}
	\newcommand{\toweakly}{\rightharpoonup}
	\newcommand{\toweaklystar}{\stackrel{*}{\rightharpoonup}}
	\newcommand{\bigO}{\mathcal{O}}

	\newcommand{\dt}{\partial_t}
	\newcommand{\grad}{\nabla}
	\newcommand{\gradvec}{\fat{\grad}}
	\newcommand{\divv}{\nabla \cdot}
	
	\newcommand{\lapl}{\Delta}
	\newcommand{\laplvec}{\mathbf{\Delta}}
	\newcommand{\unitmass}{\mathcal{M}}
	\newcommand{\unitdistance}{\mathcal{L}}
	\newcommand{\unittime}{\mathcal{T}}
	\newcommand{\unittemperature}{{\Theta}}
	\newcommand{\pr}{ \text{Pr}}
	\newcommand\fat[1]{\ThisStyle{\ooalign{
			\kern.46pt$\SavedStyle#1$\cr\kern.33pt$\SavedStyle#1$\cr
			\kern.2pt$\SavedStyle#1$\cr$\SavedStyle#1$}}}
	
	\DeclareDocumentCommand \fspace { o o m m O{} } 
	{ \IfNoValueTF {#1} 
	{ \IfNoValueTF {#2} { { {#3}^{#4}_{#5} } } {{}} {} }
	{ \if\relax\detokenize{#1}\relax { { {#3}^{#4}_{#5}}#2 } \else 
	{ \ifx\relax#2\relax ( {#3}^{#4}_{#5})^{#1}
  	\else { {#3}^{#4}_{#5} #2^{#1}} \fi }
  	\fi}
	}
	\DeclareDocumentCommand \Lp { O{} m} { \fspace[#1][]{L}{#2} }
	\DeclareDocumentCommand \Hs { O{} m} { \fspace[#1][]{H}{#2} }
	\DeclareDocumentCommand \Cs { O{} m} { \fspace[#1][]{C}{#2} }
	\DeclareDocumentCommand \Lpt{ O{} m m} { L^{#2}(#1;#3) }
	\DeclareDocumentCommand \Hst{ O{} O{} m m} { H^{#3}_{#2}(#1;#4) }
	\DeclareDocumentCommand \Cst{ O{} O{} m} { C^{#2}(#1;#3) }
    
    \DeclarePairedDelimiter\ceil{\lceil}{\rceil}

\begin{document}

\begin{frontmatter}



\title{A first order in time wave equation modeling nonlinear acoustics}


\author{Barbara Kaltenbacher, Pascal Lehner}

\affiliation{organization={Department of Mathematics},
            addressline={University of Klagenfurt}, 
            city={Klagenfurt},
            postcode={9020}, 
            state={Carinthia},
            country={Austria}}

\begin{abstract}
In this paper we focus on a small amplitude approximation of a Navier-Stokes-Fourier system modeling nonlinear acoustics. Omitting all third and higher order terms with respect to certain small parameters, we obtain a first order in time system containing linear and quadratic pressure and velocity terms.
 
Subsequently, the well-posedness of the derived system is shown using the classical method of Galerkin approximation in combination with a fixed point argument. We first prove the well-posedness of a linearized equation using energy estimates and then the well-posedness of the nonlinear system using a Newton-Kantorovich type argument. Based on this, we also obtain global in time well-posedness for small enough data and exponential decay. This is in line with semigroup results for a linear part of the system that we provide as well.

\end{abstract}

\begin{keyword}
nonlinear acoustics, well-posedness

\MSC 35L60 \sep 35L50

\end{keyword}


\end{frontmatter}


\input{Introduction}
\input{Derivation}
\input{Wellposedness}
\input{Semigroup}
\section{Conclusion}
In this paper we have shown that a first order in time quadratic wave equation can be obtained from a small amplitude approximation. We have provided the derivation of such a system from the Navier-Stokes-Fourier equations in Section \ref{sec:Derivation} and proven its well-posedness with Dirichlet boundary conditions in Section \ref{sec:Wellposedness}. 
An extension of the analysis to the practically relevant but technically somewhat more complicated Neumann and/or impedance boundary case will be subject of future research.

Since Fourier's law of heat conduction, that underlies the entropy equation \eqref{eq:energy} used here, is known to lead to nonphysical infinite speed of propagation, we plan to also investigate the use of the Maxwell-Cattaneo law and fractional versions thereof. An analysis of the resulting second order in time nonlinear wave equations can be found in, e.g., \cite{BongartiCharoenphonLasiecka20,MaxwellCattaneo,
frac_tau2zero_PartII}.
However, the derivation and analysis of pressure-velocity systems like the one studied here (but with the leading time derivative probably differing from one) is yet to be done.

From a theoretical and also practical viewpoint the limit as $\eta, \, \eps \to 0$ of \eqref{2ndOrderNSF} is of interest as well; that is, convergence to solutions of the system with $\eta=\eps=0$ in appropriate function spaces in the spirit of, e.g., \cite{BongartiCharoenphonLasiecka20,frac_tau2zero_PartII}.
This requires energy estimates that are uniform with respect to vanishing $\eta$, $\eps$ as well as investigations on their appropriate interplay when taking limits to still allow global in time well-posedness. 

The modeling and analysis results of this paper will serve as a basis for the development and analysis of numerical solution methods, in particular for adaptive space-time Galerkin methods. These are able to adequately treat shock-like scenarios and allow high parallelizability as well as $hp$-adaptivity in space-time.%
\section*{Acknowledgement}
This work was supported by the Austrian Science Fund (FWF) [10.55776/P36318]

\end{document}

%% file: Introduction.tex
\section{Introduction}
Nonlinear acoustics has recently found much interest in the physics, engineering and mathematics community. This is driven by a multitude of applications of high intensity (focused) ultrasound ranging from medical therapy and imaging to industrial cleaning and welding, requiring simulation and optimization tools that need to be well based on a mathematical analysis of the underlying system and its models. In fact, also modeling itself is a highly active field of research.
For a recent review on the mathematics of nonlinear acoustics with some (but not complete, due to ongoing activities) references we refer to \cite{Kaltenbacher:2015}. 
Details on modeling can be found, for example, in the by now standard physics references \cite{Crighton79,Hamilton:1997}.

The first order in time system to be studied here derives from the mass conservation, momentum balance and entropy equation, complemented by appropriate constitutive equations relating pressure and temperature to density and entropy. 
The decisive difference between classical fluid mechanics, where the above mentioned balance equations arise as well, and nonlinear acoustics studied here, is the typical incompressibility assumption there, while a typical assumption in acoustics is irrotationality of the particle velocity as well as its vanishing mean or background value. While recently also compressibility in fluid mechanics has been studied intensively, see., e.g., \cite{BrezinaFeireislNovotny2020} and the references therein, we derive and analyze a system that is more targeted at nonlinear acoustics and in particular contains less unknowns.

While the mathematical literature for second order in time nonlinear wave equations is meanwhile quite extensive, much less work exists so far for first order in time systems; we particularly refer to \cite{DekkersRozanova2020,Tani:2017}. 
The quest for such first order formulations arises from their usefulness for the development of efficient numerical schemes, in particular adaptive space-time discretizations, where first order systems have already been proven to serve as good starting points for discontinuous Galerkin methods in case of linear acoustics, see, e.g., \cite{DoerflerFindeisenWienersZiegler:2019, bansal:2021}. 
There the authors rely on the linearity of the underlying equations, while a more general approach can be encountered in the field of fluid dynamics. There adaptive space-time discontinuous Galerkin methods for nonlinear Navier-Stokes systems are considered, see \cite{fambri:2020} and the references therein. 
Again, the fact that the system we study here is tailored to nonlinear acoustics and only contains the velocity and the pressure as variables is supposed to allow for increased efficiency of numerical schemes to be derived for them.

\medskip

The remainder of this paper is organized as follows.

In Section~\ref{sec:Derivation}, we derive a first order in time system from the above mentioned system of balance equations and constitutive relations, skipping higher than quadratic terms and using the substitution corollary that allows to insert linearized relations into nonlinear terms -- a procedure known in nonlinear acoustics as Blackstock's scheme, cf. \cite{Blackstock:63,Lighthill:56} and see also Remark~\ref{rem:substitutioncorollary}. 
Here we base the approximation on a non-dimensionalization of the physical quantities and consideration of its asymptotics for small Mach numbers.

Section~\ref{sec:Wellposedness} studies mathematical well-posedness of the initial boundary value problem for this system of PDEs; that is, we prove existence and uniqueness of solutions and provide estimates on a certain regularity level, depending on the smoothness of the initial data and forcing functions, that need to be small enough to enable local in time well-posedness -- a fact that is also known from second order wave equation models of nonlinear acoustics and that results from potential degeneracy of the equations. We also prove global in time well-posedness and exponential decay.

A short appendix type Section~\ref{sec:Semigroup} provides some basic facts on the semigroup pertaining to a linear part of the system.


%% file: Derivation.tex
\section{Derivation of model} \label{sec:Derivation}
\newcommand{\mach}{\varepsilon}
\newcommand{\etaa}{\eta}
\newcommand{\laplnew}{\lapl'}
\newcommand{\dtnew}{\partial_{\newt}}
\newcommand{\gradnew}{\grad'}
\newcommand{\divvnew}{\gradnew \cdot}
\newcommand{\tieunit}{[t]_u}
\newcommand{\spceunit}{[x]_u}
\newcommand{\presunit}{[\pres]_u}
\newcommand{\presnew}{\pres'}
\newcommand{\velounit}{[v]_u}
\newcommand{\velonew}{\velo'}
\newcommand{\entrunit}{[\entr]_u}
\newcommand{\entrnew}{\entr'}
\newcommand{\densunit}{[\dens]_u}
\newcommand{\densnew}{\dens'}
\newcommand{\tempunit}{[\temp]_u}
\newcommand{\tempnew}{\temp'}
\newcommand{\hunit}{[h]_u}
\newcommand{\gunit}{[g]_u}
\newcommand{\lunit}{[\ell]_u}
\newcommand{\laplvecnew}{\mathbf{\Delta'}}
\newcommand\StepSubequations{
	\stepcounter{parentequation}
	\gdef\theparentequation{\arabic{parentequation}}
	\setcounter{equation}{0}
}
In this section we derive the model to be discussed in this paper from a Navier-Stokes-Fourier system, describing the motion of a viscous and heat conducting fluid, cf. \cite[Section 3]{Hamilton:1997} or \cite{DekkersRozanova2020}.
This system consists of a mass, momentum and entropy balance equation supplemented by state equations. It reads
\begin{subequations} \label{NavierStokesFourierSystem}
\begin{align} 
	\dt \dens + \divv ( \dens \velo ) &= \forcep \label{eq:mass} \\
	\dens [ \dt \velo + (\velo \cdot \grad) \velo] 
	&= - \nabla \pres + \svis \laplvec \velo + \big(\frac{\svis}{3} + \bvis \big) \grad (\divv \velo )  + \forcev \\
	\dens \temp [ \dt \entr + ( \velo \cdot \nabla ) \entr ] 
	&= \kappa \lapl \temp + \frac{\svis}{2} \big( \sum_{i,j,k} \partial_{x_j} v_i + \partial_{x_i} v_j - \frac{2}{3} \delta_{ij} \partial_{x_k} v_k \big)^2 + \bvis (\divv \velo)^2 
+\forces
\label{eq:energy}  \\
	\pres &= \pres (\dens,\entr), \quad \temp = \temp(\dens,\entr), \label{eq:state} 
\end{align}
\end{subequations}
%
%
%
where $\delta_{ij}$ is the Kronecker-Delta.
The unknown physical quantities in \eqref{eq:mass}-\eqref{NSFReducedMomentumIrrotational} are the mass density $\dens$, velocity field $\velo:=( v_1, ..., v_d)^\transpose$, specific entropy $\entr$,  pressure $\pres$ and temperature $\temp$ of the acoustic medium. All these quantities are defined on a $\dime-$dimensional domain with $\dime \in \{2 , 3\}$ and will be equipped with suitable boundary and initial conditions later on in Section~\ref{sec:Wellposedness}. 

In \eqref{eq:mass}-\eqref{NSFReducedMomentumIrrotational} the shear viscosity $\svis$, bulk viscosity $\bvis$ and thermal conductivity $\thecon= {\svis c_p}{\pr}^{-1} $ are assumed to be constant and positive. Here $\pr$ denotes the Prandtl number and $c_p$ is the specific heat capacity. 
Additionally, we consider the external force terms $\forcep$ and $\forcev= (g_1, ..., g_d)^\transpose$ that can depend on space and time, respectively.

Furthermore, since we want to model nonlinear acoustics, we make the assumption of an irrotational flow, i.e. $ \grad \times \velo = 0$. With the vector calculus identities 
\begin{align}
\begin{split}
	\grad (\divv \velo) = \laplvec \velo  + \grad \times (\grad \times \velo) \\
	(\velo \cdot \grad) \velo = \frac{\grad}{2} \big( \velo^2 \big) + (\grad \times \velo ) \times \velo
	\end{split}
\end{align}
the conservation of momentum equation thus simplifies to \addtocounter{equation}{-1}
\begin{subequations}
	\addtocounter{equation}{3}
	\begin{align}
	\dens [ \dt \velo + \frac{\grad}{2}\big( \velo^2 \big) ] 
	= - \nabla \pres + \svis \nu \laplvec \velo + \forcev \label{NSFReducedMomentumIrrotational}
	\end{align}
\end{subequations}
with $\nu:= \big(\frac{4}{3} + \frac{\bvis}{\svis} \big)$ and $\velo^2:= \velo \cdot \velo = \sum_{i=1}^d v_i^2$.

Just for reference we provide a table of units for all physical quantities that we use. In the table $\unitdistance$ stands for a unit of length, $\unittime$ for time, $\unitmass$ for mass and for $\unittemperature$ temperature.
\bgroup
{\renewcommand{\arraystretch}{1.5}
	\begin{center}
		\begin{tabular}{ |c|c|c|c|c|c|c|c|c|c|c|c| } 
			\hline
			$\dens$   & 
			$v_i$   & 
			$\entr$   & 
			$\pres$   & 
			$\temp$   & 
			$\svis$   & 
			$\bvis$   & 
			$\thecon$ &
			$\forcep$ & 
			$g_i$ & 
			$\forces$ &
			$c_p$     \\ 
			$\frac{\unitmass}{\unitdistance^3}$
			& $\frac{\unitdistance}{\unittime}$ 
			& $\frac{\unitdistance^2}{\unittime^2 \unittemperature}$ 
			& $\frac{\unitdistance^2}{\unittime^2 \unittemperature}$ 
			& $\unittemperature$
			& $\frac{\unitmass}{\unitdistance \unittime}$ 
			& $\frac{\unitmass}{\unitdistance \unittime}$  
			& $\frac{\unitmass \unitdistance}{\unittime^3 \unittemperature}$ 
			& $\frac{\unitmass}{\unittime \unitdistance^3}$ 
			& $\frac{\unitmass}{\unittime^2 \unitdistance^2}$ 
			& $\frac{\unitmass }{\unittime^3 \unitdistance}$  
			& $\frac{\unitdistance^2}{  \unittime^2 \unittemperature}$ \\
			\hline
		\end{tabular}
	\end{center}
	\egroup
In the following, we model sound waves with small amplitude that travel through some medium. We assume that the medium is in an equilibrium state before the wave propagation starts and denote by $\dens_0$, $\velo_0$, $\entr_0, \pres_0$, $\temp_0$ the mean of the corresponding quantities in equilibrium. 
Furthermore, we assume that
\begin{align} \label{EquilibriumAssumptions}
\begin{split}
    \velo_0 &= 0, \\
	\dt \dens_0 &= \dt \entr_0 =\dt \pres_0 = \dt \temp_0 = 0, \\
	\grad \dens_0 &= \grad \entr_0 = \grad \pres_0 = \grad \temp_0 = 0 .
	\end{split}
\end{align}
A small perturbation of the system is {then created} by external sources $\forcep$, $\forcev$, $\forces$ resulting in a sound wave with small amplitude. Our goal is to approximate these kinds of systems up to quadratic order in terms of small parameters characterizing the perturbation. 

The first dimensionless small parameter we consider is the so called Mach number
\begin{equation} \label{DefintionMachNumber}
	\mach:= \frac{v_r}{c_0} ,
\end{equation}
where $c_0 :=  \sqrt{ (\partial_{p } \dens)_0 }$ denotes the speed of sound in the equilibrium state, $v_r$ is a scalar reference velocity and we use the short hand notation $(\partial_{z} \cdot)_0 := \partial_{z} \cdot |_{\dens=\dens_0,s=s_0}$ with $z \in \{ \dens, s \}$.  
The second dimensionless small parameter is the inverse of the acoustic Reynolds number
\begin{equation}
\eta:=\frac{k_r \svis}{c_0 \dens_0},
\end{equation}
where $k_{r}=:\frac{\omega_{r}}{c_0}$ is a reference wave number and $\omega_r$ a reference frequency. 

In the following, we perform a change of variables in \eqref{NavierStokesFourierSystem}, which introduces a linear perturbation in terms of the Mach number $\mach$ in every unknown and makes the equation dimensionless, see e.g. \cite{Tani:2017}. 
For this we define the new variables as
\begin{align} \label{NSFVariableTransformations}
\begin{split}
\newx &= \spceunit^{-1} x, 
\quad 
\newt = \tieunit^{-1} t, 
\quad
\newr( \newx, \newt) = \frac{ \dens( \spceunit^{-1} \spacevar, \tieunit^{-1} \timevar) - \dens_0}{ \mach \densunit },
\\ 
\newv( \newx, \newt) &= \frac{ \velo( \spceunit^{-1} \spacevar, \tieunit^{-1} \timevar) } { \mach \velounit },
\quad
\newent(\newx,\newt) = \frac{ \entr( \spceunit^{-1} \spacevar, \tieunit^{-1} \timevar) - \entr_0}{ \mach \entrunit },
\\
\newp( \newx, \newt) &= \frac{ \pres( \spceunit^{-1} \spacevar, \tieunit^{-1} \timevar) - \pres_0}{ \mach \presunit }, 
\quad
\newthe(\newx,\newt) = \frac{ \temp( \spceunit^{-1} \spacevar, \tieunit^{-1} \timevar) - \temp_0}{ \mach \tempunit }, 
\\
h'(\newx,\newt) &= \frac{h( k_r \spacevar, \omega_r t)} { \mach^2 \hunit }, 
\quad 
\fat{ g}' (\newx,\newt) = \frac{\fat{ g} (k_r \spacevar,\omega_r \timevar)} { \mach^2 \gunit },
\quad
\ell'(\newx,\newt) = \frac{\ell( k_r \spacevar, \omega_r t)} { \mach^2 \lunit , 
}
\end{split}
\end{align}
%
%
%
%
with corresponding weighted units
\begin{align} \label{units}
\begin{split}
 \spceunit := \frac{1}{k_r}, 
 \quad \tieunit := \frac{1}{\omega_{r}}, 
 \quad \densunit := \dens_0, 
 \quad \velounit:=c_0, 
 \quad \entrunit:= \frac{{ \dens_0 c_0^2 }}{(\partial_{ s } p)_0},  
 \quad \presunit:= \dens_0 c_0^2, \\
\tempunit:= \dens_0 ( \partial_{\dens} \temp)_0, 
 \quad \hunit:= \tieunit^{-1} \densunit , 
 \quad \gunit:= \tieunit^{-1} \densunit \velounit , 
\quad \lunit:= \tieunit^{-1} \densunit \tempunit \entrunit.
\end{split}
\end{align}
\begin{remark}
For the derivation of the particular model studied here it is essential to scale the source terms by $\eps^2$. In this case
the linear approximation derived in Section \ref{sec:LinearApprox} describes an equilibrium state, so that when using the substitution corollary cf. Remark \ref{rem:substitutioncorollary} in the derivation of the quadratic approximation, see Section \ref{sec:QuadraticApprox}, no additional terms depending on $h'$, $\fat{g'}$, $\ell$ arise. Scaling the source terms only by $\eps$ would result in a slightly different model containing overall more  terms.
\end{remark}
\subsection{Linear Approximation} \label{sec:LinearApprox}
For the purpose of clarity {and later use in the framework of the substitution corollary, cf. Remark~\ref{rem:substitutioncorollary},} we first derive a linear approximation of \eqref{eq:mass}-\eqref{NSFReducedMomentumIrrotational} that yields a classical linear wave equation. This means that we keep only terms up to $\bigO(\eps)$ and $\bigO(\etaa)$.

Plugging \eqref{NSFVariableTransformations} into the equation of conservation of mass \eqref{eq:mass} yields
\begin{align}
   \tieunit^{-1} \dtnew (\mach \densunit \newr +  \dens_0)  + \spceunit^{-1} \divvnew ( \dens_0 \mach \velounit \newv) = \bigO(\eps^2).
\end{align}
Using \eqref{EquilibriumAssumptions}, \eqref{units} and dividing by $\eps$ the expression can be simplified to
\begin{align}
\partial_{\newt} \newr  + \divvnew \newv =  \bigO(\eps).
\end{align}
The equation of conservation of momentum \eqref{NSFReducedMomentumIrrotational} gives
\begin{align}
\dens_0  \tieunit^{-1} \dtnew (\mach \velounit \newv) = - \spceunit^{-1} \gradnew (\mach \presunit \newp + \pres_0) + \bigO(\eps^2, \eps \eta).
\end{align}
Notice that, since $\svis = \etaa \spceunit \velounit \densunit$, we have $\svis \nu \laplvec \velo  = \bigO(\eps \eta)$.
Using \eqref{EquilibriumAssumptions}, \eqref{units} and dividing by $\eps$ yields
\begin{align}
\partial_{\newt} \newv + \gradnew p' =  \bigO(\eps, \eta).
\end{align}
Finally, the entropy equation \eqref{eq:energy} gives us
\begin{align}
\begin{split}
\dens_0 \temp_0 \tieunit^{-1} \dtnew (\mach \entrunit \newent)  + (\mach \densunit \dens' \temp_0 + \dens_0 \mach \tempunit \newthe + \temp_0 \dens_0 ) \tieunit^{-1} \partial_{\newt} \entr_0 
+ \dens_0 \temp_0  (\mach \velounit \newv \cdot \spceunit^{-1} \gradnew) \entr_0 \\
= \kappa \lapl \temp_0 + \bigO( \eps^2, \eps \eta).
\end{split}
\end{align}
Here we use that $\kappa \lapl (\temp - \temp_0) = \bigO(\eps \eta)$, since $\thecon = \etaa \spceunit \velounit \densunit c_p \pr^{-1}$.
The entropy equation can therefore and due to \eqref{EquilibriumAssumptions} be written as
\begin{align} \label{eq:linear_entropy}
\partial_{\newt} s' = \bigO(\eps, \eta).
\end{align}
We proceed with simplifying the state equations \eqref{eq:state} using a Taylor 
expansion at the point $(\rho, s) = (\rho_0 , s_0)$. 
The linearized state equation of $\pres$ is
\begin{align}
\pres - \pres_0 
	=  (\partial_{\dens } p)_0 (\dens-\dens_0) + (\partial_{ s } p)_0 (s-s_0) + \bigO(\dens^2, \dens s, s^2)
\end{align}
and in dimensionless form thus reads
\begin{align}
\mach \presunit p' = \mach (\partial_{\dens } p)_0 \densunit \dens' + \mach (\partial_{ s } p)_0 \entrunit   s' + \bigO(\mach^2),
\end{align}
resulting due to \eqref{units} in, simply
\begin{align} \label{eq:linear_pressure}
    p' = \dens' + s' + \bigO(\eps).
\end{align}
Similarly for $\temp$, we have
\begin{align}
	\temp - \temp_0 =  (\partial_{\dens } \temp)_0 (\dens-\dens_0) + (\partial_{ s } \temp)_0 (s-s_0) + \bigO(\dens^2, \dens s, s^2).
\end{align}
In dimensionless form it thus reads
\begin{align}
	 \eps \tempunit \newthe =  (\partial_{\dens } \temp)_0 ( \eps \densunit \dens') + (\partial_{  s} \temp)_0 (\eps \entrunit s') + \bigO(\eps^2).
\end{align}
So by using the following thermodynamic relation, see \cite{Tani:2017},
\begin{align}
	(\partial_{ s } \temp)_0 c_0^2 \big(\frac{\gamma-1}{\gamma}\big) 
	 = ( \partial_{ s } \pres)_0 ( \partial_{  \dens} \temp)_0,
\end{align}
where $\gamma$ denotes the heat capacity ratio,
we can write
\begin{align} 
\newthe = \dens' + \frac{\gamma s'}{\gamma -1}  + \bigO(\eps) = \newp + \frac{s'}{\gamma -1}  + \bigO(\eps).
\end{align}
{Here we have used \eqref{eq:linear_pressure}} to obtain the last expression.

As a last step we formulate \eqref{eq:mass} in terms of the pressure. Combining \eqref{eq:linear_entropy} and \eqref{eq:linear_pressure} simply gives 
\begin{align}
\partial_{\newt} \dens' = \partial_{\newt} \newp + \bigO(\eps, \eta).
\end{align} 
To summarize, we simplified the Navier-Stokes-Fourier system \eqref{eq:mass}-\eqref{NSFReducedMomentumIrrotational} to
\begin{align} \label{linearNSF}
\begin{split}
\partial_{\newt} \newp  + \divvnew \newv = 0 \\
\partial_{\newt} \newv + \gradnew \newp = 0 \\
\partial_{\newt} \newent = 0 \\
\newr = \newp - \newent , \quad \newthe 
= \newp + \frac{\newent}{\gamma - 1}
\end{split}
\end{align}
up to $\bigO(\eps, \eta)$ terms.
\begin{remark}
Taking the time derivative in the first equation and the divergence in the second equation in \eqref{linearNSF} and subtracting the expressions gives the classical dimensionless wave equation
\begin{align} \label{eq:linear_wave}
\dtnew^2 \newp - \laplnew \newp = 0.
\end{align}
\end{remark}
\subsection{Quadratic Approximation} \label{sec:QuadraticApprox}
In order to obtain a quadratic approximation of \eqref{eq:mass}-\eqref{NSFReducedMomentumIrrotational} we proceed along the lines of the last subsection with the crucial difference that we now keep terms of order $\eps^2$ and $\eps \eta$. 
\begin{remark}\label{rem:substitutioncorollary}
In the following we often refer to the substitution corollary. By this we simply mean that we can substitute linear relations e.g. \eqref{linearNSF} into terms of quadratic order $\eps^2$ or $\eps \eta$, since
{no information is lost on the relevant level of order.}
\end{remark}
By performing the same change of variables \eqref{NSFVariableTransformations} the equation of conservation of mass \eqref{eq:mass} gives the following additional terms on the left hand side of the equation
\begin{align}
\eps^2 ( \spceunit^{-1} \divvnew ( \densunit \newr \velounit \newv) -  \hunit h' ),
\end{align}
so that using \eqref{EquilibriumAssumptions}, \eqref{units}, dividing by $\mach$ and the product rule yields in total
\begin{align}\label{consmom_nondim}
	\partial_{\newt} \dens' + \divvnew \newv + \eps \dens' \divvnew \newv + \eps \gradnew \dens' \cdot \newv = \eps h' + \bigO(\eps^2).
\end{align}
The equation of conservation of momentum \eqref{NSFReducedMomentumIrrotational} gives the following quadratic terms on the left hand side of the equation
\begin{align}
\begin{split}
\eps^2  ( \densunit \dens' \tieunit^{-1} \velounit \partial_{\newt}  \newv + \dens_0 \frac{\velounit^2}{2 \spceunit} {\gradnew}  \newv^2 ) 
- \eps \etaa \nu \spceunit \velounit \densunit \spceunit^{-2} \velounit \laplvecnew  \newv - \eps^2 \gunit \fat{g'} .
\end{split}
\end{align}
So once again utilizing \eqref{EquilibriumAssumptions}, \eqref{units} and dividing by $\mach$ yields
\begin{align} \label{quadratic_momentum}
	\partial_{\newt} \newv + \eps \dens' \partial_{\newt} \newv + \frac{\eps}{2} \gradnew \newv^2 + \gradnew \pres' - \eta \nu \laplvec \newv = \eps\fat{g'} + \bigO(\eps^2, \eps \eta).
\end{align}
For the entropy equation \eqref{eq:energy} we obtain the following quadratic terms on the right hand side
\begin{align}
\begin{split}
&\mach^2 ( \densunit \dens' \temp_0 + \dens_0 \tempunit \newthe ) 
\left( \tieunit^{-1} \entrunit \dtnew s' + \spceunit^{-1} \velounit  \newv \cdot \gradnew \newent_0 \right) 
+ \mach^2 \densunit \newr \tempunit \newthe \tieunit^{-1} \dtnew \newent_0 
\\ &+ \eps^2 \dens_0 \temp_0 \spceunit^{-1} ( \velounit \newv \cdot \gradnew) \entrunit \newent 
  - \mach \etaa \spceunit \velounit \densunit c_p \pr^{-1} \spceunit^{-2} \tempunit \laplnew  \newthe
-\mach^2\lunit \ell'.
\end{split}
\end{align}
The substitution corollary allows us to substitute the linear relation of $s'$, see \eqref{linearNSF}, into the $\eps^2 \dens' \partial_{\newt} \newent$ and $\eps^2 \newthe \partial_{\newt} \newent$ terms, yielding that they vanish. Using \eqref{EquilibriumAssumptions}, \eqref{units} with $\lambda:= \frac{c_p \tempunit}{\Pr \temp_0 \entrunit}$ and dividing by $\mach$ we are left with
\begin{align}
	\partial_{\newt} \newent + \eps ( \newv \cdot \gradnew) \newent = \eta \lambda \laplnew \newthe 
+\mach \ell'
+ \bigO(\eps^2, \eps \eta).
\end{align}
Using the substitution corollary once again we can substitute the linear expression for $\newthe$ from \eqref{linearNSF} into $\eta \lapl \newthe$ and substitute $s'$ in quadratic terms with some constant in time entropy function $s_l(\spacevar)$, see \eqref{linearNSF}. This gives with $\sigma:= \frac{1}{\gamma -1}$
\begin{align} \label{eq:quadratic_entropy}
	\partial_{t'} s' = - \eps ( \newv \cdot \gradnew) s_l + \eta \lambda \laplnew (p' + \sigma {s_l}) +\mach \ell'
+ \bigO(\eps^2, \eps \eta).
\end{align}
%
%
For the state equation of $\pres$ we use the Taylor series up to second order
\begin{align}
\begin{split}
	\pres - \pres_0 
	=  (\partial_{\dens } p)_0 (\dens-\dens_0) + (\partial_{ s } p)_0 (s-s_0) + (\partial^2_{ \dens } p)_0 \frac{(\dens-\dens_0)^2}{2} \\ 
	+ (\partial^2_{ s } p)_0 \frac{(s-s_0)^2}{2} + (\partial_{ \dens } \partial_{ s } p)_0  \frac{(\dens-\dens_0)(s-s_0)}{2} 
	 +\bigO(\dens^3, \dens^2 s, \dens s^2, s^3).
	 \end{split}	
\end{align}
In dimensionless form
\begin{align} \label{taylorp}
	\pres ' = \dens' + s' + \eps \frac{B}{2A} \dens'^2 + \eps \frac{C'}{2} s'^2 + \eps \frac{D'}{2} \newr s' + \bigO(\eps^2),
\end{align}
where $A,B,C',D'$ are just shorthand notations for the corresponding coefficients. The fraction $\frac{B}{2A}$ is commonly known as the coefficient of nonlinearity in the literature of nonlinear acoustics, see for example \cite[Section 2]{Hamilton:1997}. 

Applying a time derivative to \eqref{taylorp} we obtain once again due to the substitution corollary using
{$ \newr  \ \dot{=} \ \newp - \newent $, $\partial_{ t'}s' \ \dot{=} \ 0$}, $s' \ \dot{=} \ s_l$} from \eqref{linearNSF}, where $\dot{=}$ stands for a linearized relation, in combination with \eqref{eq:quadratic_entropy}
\begin{align}
	\partial_{t'} \newp  = 
	\partial_{t'} \newr - \eps (\newv \cdot \gradnew) s_l + \eta \lambda \laplnew (\newp + \sigma {s_l}) 
+\mach \ell'\\
+ \eps \frac{B}{A} (\newp - s_l) \partial_{t'} \newp + \eps \frac{D'}{2} s_l \partial_{\newt} \newp 
+ \bigO(\eps^2, \eta \eps) .
\end{align}
By inserting the linear relation $  \partial_{ t'} p' \ \dot{=} \ - \divvnew \newv$ from \eqref{linearNSF} twice, we end up with 
\begin{align}
	\partial_{ t'} \dens' = \partial_{t'} p' + \eps (\newv \cdot \gradnew) s_l - \eta \lambda \laplnew  (p' + \sigma s_l)  
-\mach \ell'\\
	+ \eps \frac{B}{A} (p'-s_l) \divvnew \newv + \eps \frac{D'}{2} s_l \divvnew \newv  
+\bigO(\eps^2, \eta \eps) .
\end{align}
Plugging {this into \eqref{consmom_nondim} with $\newr \ \dot{=} \ \newp - s_l$ from \eqref{linearNSF}} and using cancellation of the $\mach ( \newv \cdot \gradnew s_l)$ terms gives
%
%
\begin{align}
\partial_{t'} p' - \eta \lambda \laplnew p'  + \divvnew \newv 
+ \eps \frac{B}{A} (p'-s_l) \divvnew \newv + \eps \frac{D'}{2} s_l \divvnew \newv
+ \eps (\newp - s_l) \divvnew \newv + \eps \gradnew \newp \cdot \newv \\
= \eps (h'+\ell') + \etaa \lambda \sigma \laplnew s_l + \bigO(\eps^2, \eps \etaa),
\end{align}
which can be simplified to
\begin{align}
\partial_{t'} p' - \eta \lambda \lapl p'  + (1+ \eps \alpha s_l )\divvnew \newv 
+ \eps \beta p' \divv \newv 
+ \eps \gradnew \newp \cdot \newv \\
= \eps (h'+\ell') + \etaa \lambda \sigma \laplnew s_l + \bigO(\eps^2, \eps \etaa)	
\end{align}
with $\alpha:= \frac{D'}{2} - \frac{B}{A} - 1$ and $\beta:= 1 + \frac{B}{A}$.
\\
Plugging $\newr \ \dot{=} \ \newp - s_l$ and $  \partial_{ t'} \newv \ \dot{=} \ - \grad \newp$ from \eqref{linearNSF} into  $\dens' \dtnew \velo'$ from \eqref{quadratic_momentum} yields
\begin{align}
	\partial_{\newt} \newv - \eps (\newp - s_l) \gradnew \newp + \frac{\eps}{2} \gradnew (\newv^2) + \gradnew \pres' - \eta \nu \laplvec \newv = \eps\fat{g'} + \bigO(\eps^2, \eps \eta).
\end{align}
Altogether after omitting the primes we obtain 
\begin{align} \label{2ndOrderNSF}
\begin{split}
	&\dt p - \eta \lambda \lapl p  + ( 1 + \eps \alpha s_l )\divv \velo + \eps \beta p \divv \velo + \eps \grad \pres \cdot \velo 
	= \eps (h+\ell) + \etaa \lambda \sigma \lapl s_l \\ 
	&\dt \velo  - \eta \nu \laplvec \velo + (1 + \eps s_l) \grad \pres  + \frac{\eps}{2} ( \grad \velo^2 - \grad \pres^2 ) = \eps\fat{g}
\end{split}	
\end{align}
up to $\bigO(\eps^2,\eps \eta)$. After solving this system for $p$ and $v$, one can also determine the entropy $s$ from 
\begin{equation}
	\dt s = - \eps \grad s_l \cdot \velo + \eta \lambda \lapl (p + \sigma s_l)+\eps\ell,
\end{equation}
cf. \eqref{eq:quadratic_entropy}, by plain time integration. Note that the value of $s_l$ can be obtained from initial data $s_l = s'(0) = \frac{s(0) - s_0}{\eps \entrunit}$ cf. \eqref{NSFVariableTransformations}. 
\begin{remark}
If we assume $\grad s_l = 0$ we get applying $(\dt,\divv)^\transpose$ to \eqref{2ndOrderNSF} and subtracting the corresponding equations the following dimensionless version of Kuznetsov's equation, see \cite{kuznetsov:1971}, 
\begin{align}
\dt^2 \pres - (1+ \mach (\alpha+1) s_l) \lapl \pres - \eta( \lambda + \nu ) \dt \lapl \pres - \eps \frac{B}{2A} \dt^2 ({p^2}) - \eps \dt^2 ({\velo^2}) = \mach \tilde{h}
\end{align}
with $\tilde{h}= \dt \forcep - \divv \forcev$ 
combined with
$\dt \velo + \grad \pres = 0$. 

Here we used the substitution corollary with the linear relations $ \laplvec \velo  \ \dot{=} \  \partial_{t}^2 \velo \ \dot{=} \ -  \dt\grad p$, $\divv \velo \ \dot{=} \ - \dt p$, $\grad p \ \dot{=} \ - \dt \velo$, $\lapl p^2 \ \dot{=} \ \dt^2 p^2$ and $\laplvec \velo^2 \ \dot{=} \ \dt^2 \velo^2$.
\end{remark}

%% file: Wellposedness.tex
\section{Well-posedness} \label{sec:Wellposedness}
\newcommand{\lincofp}{\gamma}
\newcommand{\lincofv}{\delta}
In this section we provide an analysis of the model \eqref{2ndOrderNSF} equipped with initial and boundary conditions, that is, of 
\begin{align} \label{2ndOrderNSF_ibvp}
\begin{split}
	\dt p - \etap \lapl p + (1+ \lincofp) \divv \fat{v} + \eps_1 p \divv \fat{v} + \eps_2 \grad p \cdot \fat{v} = \forcep \\
	\dt \fat{v}  - \etav \laplvec \fat{v} + (1+ \lincofv) \grad p - \frac{\eps_3}{2} \grad (p^2) + \frac{\eps_4}{2} \grad (\fat{v}^2)= \forcev \\
	(p(0),\fat{v}(0)) = (p_0, \fat{v_0}) \\
	p = 0, \quad\fat{v} = 0  \quad \text{ on } \partial \Omega
\end{split}	
\end{align}
on a bounded $C^{1,1}$ domain $\Omega \subset \R^{d}$ with $d \in \{2,3 \}$ and over a time interval $(0,T)$.
In order not to complicate arguments too much, we use homogeneous Dirichlet boundary conditions. 

In the following we introduce the notation used in this section. Let $Z$ be a Banach space and $T\in(0,\infty)$. 
%
%
We write $n:=d+1$ and for $p \in [1,\infty], s >0$ use the short hand notation $C(T;Z):=C([0,T];Z)$ as well as,
\begin{equation}
\begin{aligned}
 \Lp[n]{p} &:=  L^p(\Omega;\R^n) \qquad &
 \Hs[n]{s} &:=  H^s(\Omega;\R^n) \\
 \Lpt[T]{p}{Z}&:=  L^p(0,T;Z) &
 \Hst[T]{s}{Z}&:=  H^s(0,T;Z) .
\end{aligned}
\end{equation}
We use $\| u \|_{ \Lp[n]{p} } := \left( {  \sum_{i=1}^n ||u_i||_{\Lp{p}}^p } \right)^{\frac{1}{p}} $
and $ \| u \|_{\Hs[n]{k} } := \left( \sum_{ |\alpha| \leq k} \| D^\alpha  u \|_{\Lp[n]{2}}^2  \right)^{\frac{1}{2}}$ as standard norms for $k \in \N$, where $D^\alpha$ stands for the $\alpha$-th weak derivative. For matrix functions we use the entry wise matrix norm $ \| u \|_{\Lp[n \times d]{p}} := \left( {  \sum_{i=1}^n \sum_{j=1}^d ||u_{ij}||_{\Lp{p}}^p } \right)^{\frac{1}{p}} $. 
Moreover, we write $\| u \|^2_{\Hs[n]{1}} = \| u \|^2_{\Lp[n]{2}} + \| \gradvec u \|^2_{\Lp[n \times d]{2}} $, where $ \gradvec u$ denotes the Jacobian of $u$.

In order to obtain $T$ independent constants it is convenient to use, in place of the standard norm on $H^1(T;Z)$, 
\begin{equation}\label{normH1T}
\|u\|_{H^1(T;Z)}:=\|u(0)\|_Z+\|\dt u\|_{L^2(T;Z)},
\end{equation}
since it allows to obtain bounds just relying on an estimate of $\dt u$ and employing the initial condition.

\medskip

We now set the stage for an analysis of \eqref{2ndOrderNSF_ibvp}.
For
\begin{equation}
u:=(p,v)^\transpose \in X, \quad
f:= ( \forcep , \forcev)^\transpose \in L^2(T;(L^2)^n), \quad 
u_0:=(p_0,\fat{v_0})^\transpose \in (H^1_0)^n, \quad
\lincofp, \lincofv \in L^\infty
\end{equation}
and $X:=X_1 \cap X_2 \cap X_3$ with
\[
\begin{aligned}
&X_1 := \Hst[T]{1}{ \Lp[n]{2} }, \quad
X_2 := \Lpt[T]{2}{ (\Hs{2} \cap H^1_0 )^n }, \quad
X_3 := \Lpt[T]{\infty}{ \Hs[n]{1} }, \\
&Y := \Lpt[T]{2}{ \Lp[n]{2} }\times (H_0^1)^n,
\end{aligned}
\]
we consider the operator $\F: X  \to Y$ defined by
\begin{align}
	\begin{split}
	{\F}(u)= {\F}(p,\fat{v}) = 
	\left(
	\begin{pmatrix}
		\dt \pres - \etap \lapl p +  (1+ \lincofp) \divv \velo + \eps_1 p \divv \fat{v} + \eps_2 \grad \pres \cdot \fat{v} - \forcep  
	 \\
		\dt \fat{v}  - \etav \laplvec \fat{v} +  (1+ \lincofv ) \grad p - \frac{\eps_3}{2} \grad (p^2) + \frac{\eps_4}{2} \grad (\fat{v}^2) - \forcev
	\end{pmatrix},
u(0)-u_0 \right)
	\end{split}	
\end{align}
for some $\etap$, $\etav$, $\eps_i>0$, $i\in\{1,\ldots, 4\}$.
As norms on $X$, $Y$ we use 
\[
\| u \|^2_X := \sum_{i=1}^3 \| u\|^2_{X_i}, \qquad
\| (f,u_0)\|_Y^2 := \|f\|_{L^2(T;(L^2)^n)}^2+\|u_0\|^2_{(H_0^1)^n}.
\] 
Note that we incorporate the 
boundary conditions via the definition of the function spaces on which $\F$ acts.

Observe that $\F$ contains linear and bilinear terms in $u$, so that we can split it accordingly. Therefore, in the following we write 
\begin{align} \label{NSFMathematicalSystem}
	\F(u) = (\Lcal u,u(0)) + (\B[u,u],0) - (f,u_0),
\end{align}
where 
\begin{align}
	\begin{split}
		{\Lcal}u := \dt u - \A u  \text{ with } 
(\A u)(t)= \AAA(u(t)), \quad 
\AAA := \begin{pmatrix}
			\etap \lapl  &-  (1+ \lincofp)\divv  
			\\
			- (1 + \lincofv) \grad  & \etav \laplvec    
		\end{pmatrix}
	\end{split}	
\end{align}  
and
\begin{align}
	\begin{split}
		&({\B}[ (p,\fat{v}), (q, \fat{w}) ])(t)=\, \BBB[(p(t),\fat{v}(t)), (q(t), \fat{w}(t))], \qquad
\BBB :=		\begin{pmatrix}
			\BBB^1 + \BBB^2
			\\
		\BBB^3 + \BBB^4
		\end{pmatrix}, \\
&\BBB[(\mathrm{p},\mathrm{v}), (\mathrm{q},\mathrm{w})]
	:= 	\begin{pmatrix}
			\eps_1 \mathrm{p} \divv \mathrm{w} + \eps_2 \grad \mathrm{q} \cdot \mathrm{v} 
			\\
			- \frac{\eps_3}{2} \grad (\mathrm{p} \mathrm{q}) + \frac{\eps_4}{2} \grad (\mathrm{v} \cdot \mathrm{w}) 
		\end{pmatrix} \qquad (\mathrm{p}, \mathrm{v})^\transpose , (\mathrm{q}, \mathrm{w})^\transpose \in (H^2)^n.
	\end{split}
\end{align} 
In the following we frequently use H\"older's inequality 
\begin{align}
	\| w v \|_{ \Lp{r} } \leq \| w \|_{ \Lp{p} } \| v\|_{ \Lp{q}} \quad w\in \Lp{p}, \ v\in \Lp{q}, \quad \frac{1}{r} = \frac{1}{p} + \frac{1}{q} 
\end{align}
for $p,q,r\in[1,\infty]$
and Young's inequality 
\begin{align}
	a b \leq \frac{\delta a^2 }{2} + \frac{b^2}{2 \delta} \quad a,b,\delta>0,
\end{align}
as well as Poincare's inequality 
\begin{align}
	\| v \|_{\Lp{2}} \leq C_P \| \grad v \|_{ \Lp[d]{2}} \quad v \in H^1_0
\end{align}
and continuity of the embeddings 
$\Hs{1} \embedded \Lp{4}$, $\Hs{2} \embedded \Lp{\infty}$ for $d \leq 3$ implying 
\[
\begin{aligned}
	&\| v \|_{\Lp{4}} \leq C_Q \| \grad v  \|_{\Lp[d]{2}} \quad &&v \in H^1_0, \\
	&\| \grad w \|_{\Lp[d]{4}} \leq C_R \| \lapl w  \|_{\Lp{2}} \quad &&w \in H^2\cap H^1_0, \\
	&\| w \|_{\Lp{\infty} } \leq C_S \| \lapl w\|_{\Lp{2}} &&w \in H^2\cap H^1_0,
\end{aligned}
\]
where in the latter case we have also combined with elliptic regularity 
\begin{align}
\| w \|_{H^2} \leq C_\lapl \| \lapl w\|_{\Lp{2}} \qquad w \in H^2\cap H^1_0.
\end{align}
Note that we use these inequalities also for vector valued functions with the same constants 
notionally. In this case $\grad$ is replaced with $\gradvec$ and $\lapl$ with $\laplvec$.

The following lemma 
is a direct consequence of the above inequalities and ensures the continuity of the quadratic terms.
\begin{lemma} \label{NSFquadraticTermEstimateLemma}
There exists $C_\B>0$ independent of $T$ such that for all $$(p,\velo)^\transpose,(q,\fat{w})^\transpose \in \tilde{X}:= X_2 \cap X_3$$
we have
\begin{align}\label{est_quadterm}
\| {\B}[ (p,\fat{v}), (q, \fat{w}) ] \|_{L^2(T;(L^2)^n)}
\leq  C_\B |\eps|  \| (p, \velo) \|_{\tilde{X}}  \| (q, \fat{w} ) \|_{\tilde{X}}
\end{align}
where $|\eps|^2:=\eps_1^2+\eps_2^2+\eps_3^2+\eps_4^2$.
\end{lemma}
\begin{proof}
We estimate the first two components of $\B$ in $ \Lpt[T]{2}{ \Lp{2} } $, which yields 
\[
\begin{aligned}
	&\frac{1}{\eps_1^2}\int_0^T \| \BBB^1 \|_{\Lp{2}}^2 dt 
	\leq \int_0^T \| \pres \|_{\Lp{4}}^2 \| \divv \welo \|_{\Lp{4}}^2 dt 
	\leq   C_Q^2 \int_0^T \| \grad \pres \|_{\Lp[d]{2}}^2 \| \gradvec \welo \|_{\Lp[d \times d]{4}}^2 dt	\\
	&\leq   C_Q^2 C_R^2  \sup_{ t \in [0,T]}\| \grad \pres \|^2_{\Lp[d]{2}} \int_0^T \| \laplvec \welo \|_{\Lp[d]{2}}^2  dt 
	\leq C_1^2 \|  (\pres, \velo) \|_{{\tilde{X}}}^2 \| (\qres, \welo) \|_{{\tilde{X}}}^2
\end{aligned}
\]
and similarly
\[
\begin{aligned}
&\frac{1}{\eps_2^2}\int_0^T \| \BBB^2 \|_{\Lp{2}}^2 dt    
\leq d  \int_0^T \| \grad \qres \|_{\Lp[d]{4}}^2 \|  \velo \|_{\Lp[d]{4}}^2 dt  
\leq d  C_R^2 \int_0^T \| \lapl \qres \|_{\Lp{2}}^2 \|  \velo \|_{\Lp[d ]{4}}^2 dt  \\
&		\leq d  C_Q^2 C_R^2  \sup_{ t \in [0,T]} \| \gradvec \velo \|^2_{\Lp[d\times d]{2}} \int_0^T \| \lapl \qres \|_{\Lp{2}}^2 dt
		\leq C_2^2 \|  (\pres, \velo) \|_{{\tilde{X}}}^2 \| (\qres, \welo) \|_{{\tilde{X}}}^2.
\end{aligned}
\]
The last two terms are estimated in $\Lpt[T]{2}{(L^2)^d}$ giving
\[
\begin{aligned}
&\frac{4}{\eps_3^2} \int_0^T \| \BBB^3 \|_{\Lp[d]{2}}^2 dt 
		\leq  d \int_0^T \Bigl(\| \grad \pres \|_{\Lp[d]{4}}^2 \|  \qres \|_{\Lp{4}}^2 + \| \grad \qres \|_{\Lp[d]{4}}^2 \|  \pres \|_{\Lp{4}}^2\Bigr)  dt \\
&		\leq d C_Q^2 C_R^2 \Bigl( \sup_{t \in  [0,T]} \|  \grad \qres \|_{ {\Lp[d]{2}} }^2\int_0^T \| \lapl \pres \|_{\Lp{2}}^2  dt + \sup_{t \in [0,T]}  \|  \grad \pres \|_{{\Lp[d]{2}}}^2 \int_0^T \| \lapl \qres \|_{\Lp{2}}^2   dt \Bigr)\\
&		\leq 4 C_3^2 \|  (\pres, \velo) \|_{\tilde{X}}^2 \| (\qres, \welo) \|_{{\tilde{X}}}^2
\end{aligned}
\]
and similarly
\[
\begin{aligned}
&	\frac{4}{\eps_4^2} \int_0^T \| \BBB^4 \|_{\Lp[d]{2}}^2 dt 
		\leq d \sum_{i=1}^d \int_0^T \Bigl(\| \grad \velo_i \welo_i  \|_{\Lp{2}}^2 + \| \grad \welo_i \velo_i  \|_{\Lp{2}}^2\Bigr)  dt \\
&		\leq d^2 C_Q^2 C_R^2  \Bigl( \sup_{t \in  [0,T]} \|  \gradvec \welo \|^2_{\Lp[d \times   d]{2}} \int_0^T \| \laplvec \velo \|_{\Lp[d]{2}}^2  dt +  \sup_{ t \in [0,T]} \|  \gradvec \velo \|^2_{\Lp[d \times   d]{2}} \int_0^T \| \laplvec \welo \|_{\Lp[d]{2}}^2 dt \Bigr) \\
&		\leq  4 C_4^2 \|  (\pres, \velo) \|_{{\tilde{X}}}^2 \| (\qres, \welo) \|_{{\tilde{X}}}^2 .
\end{aligned}
\]
Taking square roots and summing up the estimates we obtain \eqref{est_quadterm} with 
\\
$C_\B:= $ $4 \max \{C_1,C_2, C_3, C_4\}$.
\end{proof}
Clearly, the above Lemma remains valid with $\tilde{X}$ replaced by $X$.
With this result at hand, Fr\'{e}chet differentiability of $\F$ and Lipschitz continuity of the Fr\'{e}chet derivative are straightforward.
\begin{lemma} \label{NSFDifferentaibleLemma}
	The Fr\'{e}chet derivative of $\F:X\to Y$ at $u \in X$ is given by 	
\begin{align}\label{Fprime}
	\Fprime{u}{h} = (\Lcal h + \B[u,h] + \B[h,u],h(0))
	\end{align}
for $h \in X$. 
Moreover, the Lipschitz estimate 
\begin{align}\label{LipschitzFprime}
	\|\Fprime{u_1}{h}-\Fprime{u_2}{h}\|_Y\leq 2C_\B|\eps| \|u_1-u_2\|_{\tilde{X}} \|h\|_{\tilde{X}} 
	\end{align}
holds.
\end{lemma}
\begin{proof}
	We check that each term is Fr\'{e}chet differentiable. For the linear terms this is clear. For the quadratic terms we can use the bilinearity in combination with Lemma \ref{NSFquadraticTermEstimateLemma} and get for $h \in X$ with $\tilde{Y}:=L^2(T;(L^2)^n)$
\begin{align}
\begin{split}
 \| \B[u + h, u + h] - \B[ u, u ] - \B[u, h] - \B[h , u] \|_{\tilde{Y}}   
= \| \B[h,h] \|_{\tilde{Y}} \leq C_\B|\eps| \| h \|_X^2. 
\end{split}
\end{align}
Thus with $\Fprime{u}{h}$ as in \eqref{Fprime},
\begin{align}
\frac{ \| \F(u+h) - \F(u) - \Fprime{u}{h} \|_Y}{\|h\|_X} \leq C_\B |\eps| \| h\|_X \to 0 
\text{ as }\|h \|_X \to 0.
\end{align}
The Lipschitz estimate \eqref{LipschitzFprime} follows from the identity  
\[
\B[u_1,h] + \B[h,u_1] - \B[u_2,h] - \B[h,u_2]
=\B[u_1-u_2,h] + \B[h,u_1-u_2]
\]
and Lemma~\ref{NSFquadraticTermEstimateLemma}.
\end{proof}
Note that with $h:=(h_p, \fat{h_v})^\transpose$ we have
\begin{align}
\B[u,h] + \B[h,u] =
\begin{pmatrix}
			\eps_1 (  h_p \divv \fat{v} + p \divv \fat{  h_v }) + \eps_2 (\grad  h_p \cdot \fat{v} + \grad p \cdot \fat{ h_v} ) \\
			- {\eps_3} \grad (p  h_p) + {\eps_4} \grad (\fat v \cdot \fat{ h_v}) 
		\end{pmatrix}	.
\end{align}

\medskip

Our goal is to prove solvability of $\F(u)=0$ by using a Newton-Kantorovich type fixed point argument, cf. e.g., \cite{Ortega1968}. 
Thus, in order to define a fixed point equation for $u$, we consider the Newton step starting from some point $u_*$, which is formally given by 
\begin{align}
u = u_*  -  \F'^{-1}_{u_*} \F(u_*).
\end{align}
Using Lemma \ref{NSFDifferentaibleLemma}, 
this can be reformulated as
\[
0 =	\Fprime{u_*}{u} - \Fprime{u_*}{u_*} +	{\F}(u_*)  
%
= (\Lcal u + \B[u_*,u] + \B[u,u_*] - \B[u_*,u_*] - f, u(0)-u_0)
\]
or in short
\begin{align} \label{linear system newton}
	u =  \mathcal{T}(u_*):=\F'^{-1}_{u_*} \tilde{f}(u_*)
\end{align}
with
\begin{align}\label{ftilde}
	\tilde{f}(u_*) := (\B[u_*,u_*] + f,u_0) = \left(\begin{pmatrix}
		\eps_1 p_* \divv \fat{v_*} + \eps_2 \grad p_* \cdot \fat{v_*} \\
		- \frac{\eps_3}{2} \grad (p_*^2) + \frac{\eps_4}{2} \grad(\fat{v_*}^2) 
	\end{pmatrix} + f, \ u_0\right),
\end{align}
which we regard as the linearized version of system \eqref{NSFMathematicalSystem}.

The key steps to make the fixed point formulation \eqref{linear system newton} rigorous (and with $u=u_*$ equivalent to $\F(u)=0$) are to show that 
\begin{itemize}
\item $\F'_{u_*}$ is boundedly invertible $ \F'^{-1}_{u_*} \in \linearspace(Y,X)$
for all $u_*$ sufficiently small, which we state in Theorem~\ref{Theorem Linearized wellposed} and prove in Section~\ref{sec:proofthmlin}; 
\item $\tilde{f}(u_*)$ is well-defined by \eqref{ftilde} for all $u_*$ sufficiently small, which we state in the next lemma.
\end{itemize}

As an immediate consequence of Lemma~\ref{NSFquadraticTermEstimateLemma} and \eqref{ftilde} we have the following.
\begin{lemma} \label{Lemma f tilde}
	Given $u_*\in X$ and $(f,u_0) \in Y$ we have that $\tilde{f}(u_*) \in Y$ with  
	\begin{align}
		\| \tilde{f}(u_*) \|^2_Y \leq { (C_\B |\eps| \| u_* \|_X^2 + \| f\|_{L^2(T;(L^2)^n)} )^2
+\|u_0\|^2_{(H_0^1)^n} }.
	\end{align}
\end{lemma}

To show bounded invertibility of $\F'_{u_*}$, we will need to impose smallness of $u_*$. For this purpose we denote the standard balls in the $X$-norm by 
\[
\begin{aligned}
&B^{X}_{r}(u_*) :=  \{ u \in X : \| u -u_*\|_{X} \leq r\}, \quad 
B^{X}_{r} := B^{X}_{r}(0), \\
&U^{X}_{r}(u_*) :=  \{ u \in X : \| u -u_*\|_{X} < r\}, \quad 
U^{X}_{r} := U^{X}_{r}(0) 
\end{aligned}
\]
for some $r>0$.
\begin{theorem}[Linearized well-posedness] \label{Theorem Linearized wellposed}
	Let $T>0$. Then there exist $r, \tilde{r} >0$ and $C_G>0$ independent of $T$ such that for each $u_* \in B^{X}_{r}$, $(f,u_0) \in Y$ and $\lincofp, \lincofv \in B_{\tilde{r}}^{L^\infty}$ there is a unique $u \in X$ with 
	\begin{align} \label{WeakSolution}
	\Fprime{u_*}{u} = (f,u_0) 
	\end{align}
	and $u$ satisfies the a priori estimate 
	\begin{align}\label{apriori}
		\| u \|_X \leq C_G \|(f,u_0)\|_Y .
	\end{align}
\end{theorem}
%
%
\subsection{Proof of Theorem~\ref{Theorem Linearized wellposed}}\label{sec:proofthmlin}
The proof is split into several steps and uses the classical method of Galerkin approximation, see cf., eg., \cite[Section 7]{Evans:2010}. 
\subsubsection{Galerkin Approximation}
\newcommand{\galcoff}{D}
\newcommand{\galcofff}{d}
\newcommand{\galeig}{\sigma}
\newcommand{\galfun}{u}
\newcommand{\galspa}{U}
\newcommand{\odea}{M}
\newcommand{\odeb}{b}
For the following Galerkin method we make use of the family of eigenfunctions $\{ \galeig^k \}_{k=1}^\infty$ of the Dirichlet Laplacian $- \Delta:H^2 \cap H^1_0 \to L^2$. 
We can normalize $\{ \galeig^k \}_{k=1}^\infty$ such that it is an orthogonal basis of $H^1_0$ and an orthonormal basis of $L^2$. 

For any fixed $m \in \N$ we define $\sigma_m:=(\sigma^1 , \dots, \sigma^m)^\transpose \in (H^2\cap H_0^1)^m$ and define the solution ansatz as
\begin{align} \label{GalerkinAnsatzFunctions}
\begin{split}
\galfun_m(t) := \galcoff_m(t) \galeig_m \in \galspa_m := (span \{ \sigma^1 , \ldots, \sigma^m \} )^n
\end{split}
\end{align}
For a suitable choice of the coefficient matrix $\galcoff : [0,T] \to \R^{n \times m}$ we want to solve the projection of $\Fprime{u_*}{\galfun_m} = f$ onto the finite dimensional subspace $\galspa_m$. 
That is, we are looking for a function $D_m$ such that $\galfun_m$ according to \eqref{GalerkinAnsatzFunctions} solves 
\begin{align} \label{GalerkinPDE}
\begin{split}
	\pro_{\galspa_m}( \dt u_m(t) -\AAA u_m(t)+ \BBB[u_m(t),u_*(t)] + \BBB[u_*(t),u_m(t)] - f(t) ) &=0 \\
	u_m(0) &= 	\pro_{\galspa_m} (u_0) 
\end{split}
\end{align}
for a.e $t \in [0,T]$, where $\pro_{\galspa_m}$ denotes the component wise projection from $L^2$ onto $\galspa_m$.
Note that the boundary conditions are automatically satisfied due to the choice of the ansatz functions.
\begin{lemma}
Let $m \in \N$, $T\in(0,\infty)$, $\gamma, \delta \in L^\infty$, $({f},u_0) \in Y$ and $u_* \in X $. 
Then there exists a unique absolutely continuous function $\galcoff_m:[0,T] \to \R^{n \times m}$ that solves \eqref{GalerkinPDE} with \eqref{GalerkinAnsatzFunctions}.
\end{lemma}
\begin{proof}
For the sake of notation, in the following we consider the vectorized version $\galcofff_m$ of $\galcoff_m$ defined as
\begin{align}
\galcofff_m := ( \galcoff_m^{1,1} , \dots , \galcoff_m^{1,m}  ,\dots,\galcoff_m^{n,1}, \dots \galcoff_m^{n,m} )^\transpose = (d^1, \dots, d^{nm})^\transpose \in \R^{nm}.
\end{align} 
Testing the $l$-th equation of \eqref{GalerkinPDE} with $\galeig^j$ for each $l \in \{1, \dots, n\}$ and $j \in \{ 1, \dots, m\}$ gives
the following initial value problem for an $(n m)$-dimensional nonhomogeneous linear ODE with time dependent coefficients
\begin{align} \label{GalerkinODE}
\begin{split}
	\dot{\galcofff}_m(t) &=  \odea_m(t) \galcofff_m(t) + \odeb_m(t) \\
	\galcofff_m(0)  &= \int_\Omega ( (u_0)_1 \galeig^1, \dots, (u_0)_1 \galeig^m, \dots, (u_0)_n \galeig^1, \dots, (u_0)_n \galeig^m)^\transpose
\end{split}
\end{align}
with
\begin{align}
\begin{split}
	\odea^{i,j}_m(t) &:= \sum_{l=1}^n \sum_{k_l=(l-1)m+1}^{lm} \!\!\!\!\!\!\!\!\delta_{k_l j} \int_\Omega  
(\AAA-\BBB[u_*(t),\cdot]-\BBB[\cdot,u_*(t)])_{\chi(i)}( \galeig^{k_1 \bmod m}\!\!\!\!\!\!\!\!, \ldots, \galeig^{k_n\bmod m})   
	\galeig^{i \bmod m}   \\
 	\odeb_m^{i}(t) &:= \int_{\Omega} {{f}(t)_{\chi(i)} } \galeig^{i \bmod m},
 	\end{split} 
\end{align}
where $i \in \{ 1, ..., nm\}$ and $\chi(i)= \ceil*{\frac{i}{m}}$ with $\ceil*{\cdot}$ denoting the ceil function that rounds up to the next integer. Here we use entry wise integration if the integral sign is in front of a vector. 

Indeed, using the linearity of $\AAA-\BBB[u_*(t),\cdot]-\BBB[\cdot,u_*(t)]$ in each component we get 
\[
\begin{aligned}
&(\odea_m(t) d_m(t))_i = \sum_{j=1}^{nm} \odea^{i,j}_m(t) d^j_m(t)  \\
&=   \sum_{l=1}^n \sum_{k_l=(l-1)m+1}^{lm} \int_\Omega  (\AAA-\BBB[u_*(t),\cdot]-\BBB[\cdot,u_*(t)])_{\chi(i)}( \galeig^{k_1 \bmod m}\!\!\!\!\!\!\!\!, \ldots, \galeig^{k_n\bmod m}) \galeig^{i \bmod m} d^{k_l}_m   \\
&=   \sum_{l=1}^n \sum_{k_l=(l-1)m+1}^{lm} \int_\Omega  (\AAA-\BBB[u_*(t),\cdot]-\BBB[\cdot,u_*(t)])_{\chi(i)}( d_m^{k_1} \galeig^{k_1 \bmod m}\!\!\!\!\!\!\!\!, \ldots, d_m^{k_n}\galeig^{k_n\bmod m}) \galeig^{i \bmod m} \\
&=   \int_\Omega  (\AAA-\BBB[u_*(t),\cdot]-\BBB[\cdot,u_*(t)])_{\chi(i)}( \sum_{k_1=1}^{m} d_m^{k_1} \galeig^{k_1 \bmod m}\!\!\!\!\!\!\!\!, \ldots, \!\!\!\!\!\!\!\!\sum_{k_n=(n-1)m+1}^{nm}d_m^{k_n}\galeig^{k_n\bmod m}) \galeig^{i \bmod m} \\
&=   \int_\Omega  (\AAA-\BBB[u_*(t),\cdot]-\BBB[\cdot,u_*(t)])_{\chi(i)}( \galfun_m) \galeig^{i \bmod m},
\end{aligned}
\] 
which corresponds to testing the vectorized version of \eqref{GalerkinPDE}.
 
For the ODE in \eqref{GalerkinODE} we have the following properties. \\
i) System \eqref{GalerkinODE} is linear in $\galcofff_m$, thus clearly also Lipschitz continuous.  \\ 
ii) $\odeb_m$ can be estimated using H\"older's inequality and $\| \galeig^{i \bmod m} \|_{\Lp{2}} = 1$ 
\begin{align}
 \int_0^T | \odeb^{i}_m(t)| dt   
\leq  \int_0^T \| {f}_{\chi(i)}(t) \|_{\Lp{2}} \| \galeig^{i \bmod m} \|_{\Lp{2}} dt \leq  \| {f} \|_{L^1(T;(L^2)^n)} \leq \sqrt{T} \| {f} \|_{L^2(T,(L^2)^n)},
\end{align}
so that $\odeb_m \in L^1([0,T])^{n m}$, where we use $T<\infty$. 
Similarly, for $\odea_m$ and $\fat{k}=(k_1,\ldots,k_n)\in \N^n$, $\galeig^{\fat{k}}=(\galeig^{k_1},\ldots,\galeig^{k_n})$
we have with H\"older's and triangle inequality
\[
\begin{aligned}
&\int_0^T | \odea_m^{i,j}(t)| dt  
  \leq \sum_{ |\fat{k}|_\infty  \leq n}  \int_0^T \| \AAA_{\chi(i)}(\galeig^{\fat{k}}) \|_{\Lp{2}} + \| \BBB_{\chi(i)}[\galeig^{\fat{k}}, u_*(t)]\|_{\Lp{2}} +  \| \BBB_{\chi(i)}[u_*(t),\galeig^{\fat{k}}]\|_{\Lp{2}} dt \\
 & \leq  \sum_{ |\fat{k}|_\infty  \leq n} C_{\AAA}   T \| \galeig^{\fat{k}} \|_{\Hs[n]{2}} + 2 \max\{\eps_1,\ldots,\eps_4\} C_R^2 \int_0^T \| u_*(t)\|_{\Hs[n]{2}} \| \galeig^{\fat{k}} \|_{\Hs[n]{2}} dt \\
&\leq C(\galspa_m,\etap,\etav, \gamma, \delta, \eps,T) (1+\| u_* \|_X),
\end{aligned}
\]
where we have estimated the $\BBB$ terms like in the proof of Lemma \ref{NSFquadraticTermEstimateLemma} and write $ |\fat{k}|_\infty = \max_{ i \in \{1, \dots, n\}} k_i$. Here 
$C_\AAA := 1 + \max\{ \etap,\etav, \| \lincofp \|_{L^\infty}, \| \lincofv \|_{L^\infty} \}$.
This allows us to conclude that $\odea_{m}(t)$ and $\odeb_m(t)$ are in $L^1([0,T])^{nm \times nm}$ and $L^1([0,T])^{nm}$ respectively, so that $V_m:\,(t,d)\mapsto\odea_m(t) d + \odeb(t)$ defines a Carathr\'{e}odory function. \\
iii) Moreover, we have the following growth condition for $V_m$
\begin{align}
	| \odea_m(t) d + \odeb_m(t)| \leq C |\odea_m(t)| |d| + |\odeb(t)| ,
\end{align}
where we have $t\mapsto C|\odea_m(t)| \in L^1([0,T])$ for some $C>0$.
By \cite[Theorem 1.45]{Roubicek:2013} we obtain existence of a unique absolutely continuous solution $\galcoff_m \in W^{1,1}([0,T])^{n \times m}$.
\end{proof}
\subsubsection{Energy estimates}
In the following we derive energy estimates for the Galerkin approximation \eqref{GalerkinPDE} for fixed $m \in \N$.
We make use of the following integration by parts formulas, assuming that $u,v \in H^2(\Omega)$ and $w \in H^1(\Omega)$, $U \in H^1(\Omega)^d$. 
\begin{align} \label{partial integration formulas}
	\begin{split}
		\int_{\Omega} w (\divv U)  &= - \int_{\Omega} \grad w \cdot U + \int_{\partial \Omega} w U \cdot \fat{n} dS \\
		\int_{\Omega} u \lapl v &= - \int_{\Omega} \grad u \cdot \grad v + \int_{\partial \Omega} u \grad v \cdot \fat{n} dS,
	\end{split}
\end{align}
where $\fat{n} \in L^1(\partial \Omega)^d$ is the outer unit normal vector of $\Omega$.
Note that for every $m \in \N$ we have that $(p_m(t),\fat{v_m}(t))^\transpose \in (H^2 \cap H_0^1)^n$ for a.e. $t$, which allows us to use the above formulas in the following. 

We can prove the following energy estimate.
\begin{lemma}\label{lem:enest}
There exist $r, \tilde{r}>0$ and $C_G>0$ independent of $T$ and $m$ such that for $u_* \in B_{r}^{\tilde{X}}$ and any $\lincofp, \lincofv \in B_{\tilde{r}}^{L^\infty}$, $(f,u_0) \in Y$ the Galerkin approximation $u_m$ solving \eqref{GalerkinPDE} satisfies the a priori estimate
\begin{align}\label{enest_Galerkin}
	\| u_m \|_{X} \leq C_G  \| (f,u_0) \|_Y .
\end{align}
\end{lemma}
Note that we need smallness of $u_*$ only in the weaker norm of $\tilde{X}$.
\begin{proof}
1. A lower order energy identity can be obtained by testing system \eqref{GalerkinPDE} with $u_m(t)=(p_m(t),v_m(t))^\transpose \in \galspa_m$. After skipping the subscript $m$ for better readability, this yields using integration by parts
\begin{equation}\label{enid0}
\begin{aligned}
0=&	\langle \dt u_m(t) -\AAA u_m(t)+ \BBB[u_m(t),u_*(t)] + \BBB[u_*(t),u_m(t)] - f(t),u_m(t)\rangle_{(L^2)^n}\\ 
& =\frac12 \frac{d}{dt} \|u(t)\|_{(L^2)^n}^2+\etap\|\grad p(t)\|_{(L^2)^d}^2+\etav \|\gradvec \fat{v}(t)\|_{(L^2)^{d \times d}}^2  \\
&\qquad
+ \langle (\lincofp \divv \fat{v}(t) , \lincofv \grad p(t) )^\transpose + \BBB[u(t),u_*(t)] + \BBB[u_*(t),u(t)] - f(t),u(t)\rangle_{(L^2)^n}\\
&\geq \frac12 \frac{d}{dt} \|u(t)\|_{\Lp[n]{2}}^2 + \frac{\etap - c}{C_P^2} \|p(t)\|_{\Hs{1}}^2+  \frac{\etav - c}{C_P^2} - \|\fat{v}(t)\|_{\Hs[n]{1}}^2\\
&\qquad- \bigl(\|\BBB[u(t),u_*(t)]\|_{\Lp[n]{2}} + \|\BBB[u_*(t),u(t)]\|_{\Lp[n]{2}} + \|f(t)\|_{\Lp[n]{2}}\bigr)\|u(t)\|_{\Lp[n]{2}}.
\end{aligned}
\end{equation}
Here we have used the Cauchy-Schwarz inequality in space and cancellation within the skew symmetric $\divv$ and $\grad$ terms of $\AAA$
\[
\int_\Omega \Bigl( p(t) \divv \fat{v}(t)  \,  + \grad p(t) \cdot \fat{v}(t)\Bigr)\, dx =0
\]
and
\[
C_P^2 \int_\Omega \Bigl| \lincofp \divv \fat{v}(t)  \, p(t) + \lincofv \grad p(t) \cdot \fat{v}(t)\Bigr|\ dx 
\leq c ( \| p(t) \|^2_{H^1} + \| \velo(t) \|^2_{(H^1)^n} )
\]
with $c := \frac{C_P^2}{2} ( \| \lincofp \|_{L^\infty} + d \| \lincofv \|_{L^\infty})$.
Integrating from $0$ to $s$ for $s \in [0,T]$ and using Poincar\'{e}'s inequality, the Cauchy-Schwarz inequality in time, as well as Lemma~\ref{NSFquadraticTermEstimateLemma} gives
\begin{equation}\label{enest0}
\begin{aligned}
	&\frac{1}{2} \Bigl( \| u(s) \|^2_{\Lp[n]{2}} - \| u_0 \|^2_{\Lp[n]{2} } \Bigr) + 	
\frac{\min\{\etap ,\etav  \} - c}{C_P^2} \int_0^s \| u\|_{\Hs[n]{1}}^2   
	\\
	&\leq \Bigl(2 C_\B | \eps| \|u\|_{\tilde{X}} \|u_*\|_{\tilde{X}} + \|f\|_{L^2(T;\Lp[n]{2})}\Bigr)
\|u\|_{L^2(T;\Lp[n]{2})},
\end{aligned}
\end{equation}
where we choose \begin{equation} \label{gammdeltasmallness}
\tilde{r} := \frac{ \min \{ \mu, \eta \}}{C_P^2 d}, \end{equation} so that $c< \min \{ \mu, \eta \}$. \\
2. Similarly, a higher order energy identity results from testing \eqref{GalerkinPDE} with $-\lapl u_m(t)$  $=(-\lapl p_m(t),-\laplvec \velo_m(t))^\transpose \in \galspa_m$ as follows. Again skipping the subscripts after the first line and applying integration by parts gives 
\begin{equation}\label{enid1}
\begin{aligned}
0=&	\langle \dt u_m(t) -\AAA u_m(t)+ \BBB[u_m(t),u_*(t)] + \BBB[u_*(t),u_m(t)] - f(t), -\lapl u_m(t)\rangle_{(L^2)^n}\\ 
& =\frac12 \frac{d}{dt} \|\grad u(t)\|_{(L^2)^{n \times d}}^2+\etap\|\lapl p(t)\|_{L^2}^2+\etav \|\laplvec \fat{v}(t)\|_{(L^2)^n}^2\\ 
&\qquad
-\int_\Omega \Bigl((1+\lincofp)\divv \fat{v}(t)  \lapl p(t) + (1+\lincofv) \grad p(t) \cdot \laplvec \fat{v}(t)\Bigr)\, dx\\ 
&\qquad
+ \langle\BBB[u(t),u_*(t)] + \BBB[u_*(t),u(t)] - f(t),-\lapl u(t)\rangle_{(L^2)^n}\\
\end{aligned}
\end{equation}
Differently from \eqref{enid0}, we make no attempt to use cancellations of the skew symmetric terms from $\AAA$ since this would lead to additional boundary terms.
Instead, we estimate them using the Cauchy-Schwarz inequality in space and time in combination with Young's inequality  
\newcommand{\myC}{C_1}
\newcommand{\myCC}{C_2}
\[
\begin{aligned}
&\left|\int_0^s\int_\Omega \Bigl( (1+\lincofp)\divv \fat{v}(t) \, \lapl p(t) + (1+\lincofv) \grad p(t) \cdot \laplvec \fat{v}(t)\Bigr)\, dx\right|\\
&\leq (1 + \| \lincofp \|_{L^\infty}) \|\divv \fat{v}\|_{L^2(s;\Lp{2})} \| \lapl p\|_{L^2(s;\Lp{2})} 
+ (1+ \| \lincofv \|_{L^\infty}) \|\grad p\|_{L^2(s;\Lp[n]{2})} \|\laplvec \fat{v}\|_{L^2(s;\Lp[n]{2})}
\\
&\leq \myC  \| u \|^2_{L^2(s;(H^1)^n)} 
+ \frac{\min\{\etap,\etav\}}{2 C_\lapl^2}\| u \|^2_{L^2(s;(H^2)^n)} 
\end{aligned}
\]
with $\myC := (1+ \frac{\| \lincofp \|_{L^\infty} + \| \lincofp \|_{L^\infty}}{2})^2 \frac{C_\lapl^2}{\min \{ \mu, \eta\}} $.
Inserting this into the time integrated version of \eqref{enid1}, using Lemma~\ref{NSFquadraticTermEstimateLemma} and absorbing the $\| u \|_{L^2(s;(H^2)^n)}$ term gives
\begin{equation}\label{enest1}
\begin{aligned}
&	\frac{1}{2} \Bigl( \| \nabla u(s) \|^2_{\Lp[n\times d]{2}} - \| \nabla u_0 \|^2_{\Lp[n\times d]{2} } \Bigr) + 	
\frac{\min\{\etap,\etav\}}{2 C_\lapl^2} \int_0^s \| u\|_{\Hs[n]{2}}^2   
\\
&	\leq 
\Bigl(2 C_\B | \eps| \|u\|_{\tilde{X}} \|u_*\|_{\tilde{X}} + \|f\|_{L^2(T;\Lp[n]{2})}\Bigr)
\|u\|_{L^2(T;\Hs[n]{2})} + \myC \| u\|_{L^2(s;\Hs[n]{1})}^2 .
\end{aligned}
\end{equation}
The last term can be controlled by the left hand side of  \eqref{enest0}. 
This leads us to combining \eqref{enest0}+ $ \myCC$\eqref{enest1} with $\myCC:=\frac{\min\{\etap ,\etav  \} - c}{ 2 C_P^2 \myC}$ to obtain, for the energy
\begin{equation}\label{energy}
\mathcal{E}[u](s):=  \frac{1}{2} \| u(s) \|^2_{\Lp[n]{2}} 
+ \frac{\myCC}{2} \| \gradvec u(s) \|^2_{\Lp[n\times d]{2}} 
\end{equation}
the estimate
\newcommand{\cyoung}{\sigma}
\begin{equation}\label{enest}
\begin{aligned}
&\mathcal{E}[u](s)
+ \frac{ \min\{\etap ,\etav  \} - c}{ 2 C_P^2} \int_0^s \| u\|_{\Hs[n]{1}}^2   
+ \frac{ \myCC \min\{\etap,\etav\}}{  2 C_\lapl^2} \int_0^s \| u\|_{\Hs[n]{2}}^2   
	\\
&	\leq 
\Bigl(2 C_\B |\eps| \|u\|_{\tilde{X}} \|u_*\|_{\tilde{X}} + \|f\|_{L^2(T;\Lp[n]{2})}\Bigr) 
	\Bigl(\myCC \|u\|_{L^2(T;\Hs[n]{2})}
+\|u\|_{L^2(T;\Lp[n]{2})}\Bigr)
+\mathcal{E}[u](0)
	\\
&	\leq 
\Bigl(2 C_\B | \eps | \|u\|_{\tilde{X}} \|u_*\|_{\tilde{X}} + \|f\|_{L^2(T;\Lp[n]{2})}\Bigr)
\Bigl( \myCC +1\Bigr)\|u\|_{\tilde{X}}
+\mathcal{E}[u](0)
	\\
&	\leq 
\Bigl(\myCC +1\Bigr)
\Bigl\{\Bigl(2 C_\B |\eps| \|u_*\|_{\tilde{X}}+\frac{\cyoung}{2}\Bigr) \|u\|_{\tilde{X}}^2
+ \frac{1}{2\cyoung}\|f\|_{L^2(T;\Lp[n]{2})}^2\Bigr\}
+\mathcal{E}[u](0)
,
\end{aligned}
\end{equation}
where we have used Young's inequality to obtain the last inequality with $\cyoung>0$ yet to be chosen.
The left hand side of \eqref{enest}, after taking the supremum over $s$, due to  
\[
\sup_{s\in[0,T]} \bigl( \sum_{i=1}^4 a_i(s)\bigr)
\geq \frac{1}{4} \sum_{i=1}^4 \bigl(\sup_{s\in[0,T]} a_i (s)\bigr)
\]
for $a_i \geq 0$ is an upper bound for the full $\tilde{X}$ norm of $u$, i.e. we have
\[
\sup_{s\in[0,T]}\Bigl\{\mathcal{E}[u](s)
+ \frac{ \min\{\etap ,\etav  \} - c}{ 2 C_P^2} \int_0^s \| u\|_{\Hs[n]{1}}^2   
+ \frac{ \myCC \min\{\etap,\etav\}}{  2 C_\lapl^2} \int_0^s \| u\|_{\Hs[n]{2}}^2     \Bigr\}
\geq \underline{c} \|u\|_{\tilde{X}}
\]
cf, \eqref{energy}, with $\underline{c}= \frac{1}{4} \min\{
\frac{1}{2}, \frac{C_2}{2}, 
\frac{ \min\{\etap ,\etav  \} - c}{ 2 C_P^2}, 
\frac{ \myCC \min\{\etap,\etav\}}{  2 C_\lapl^2}
\}$.
Assuming $u_*$ to be small and choosing $\cyoung>0$ small enough  
\begin{equation}\label{smallness_ustar} 
 \|u_*\|_{\tilde{X}}\leq r:=\frac{\underline{c}}{8 C_\B |\eps|} \Bigl(\myCC +1\Bigr)^{-1}, \quad 
\cyoung := \frac{\underline{c}}{2} \Bigl(\myCC +1\Bigr)^{-1}
\end{equation}
we obtain (recalling that $u=u_m$ is the Galerkin approximation here)
\begin{align}\label{enest_Galerkin_tilde}
	\| u_m \|_{\tilde{X}} \leq \tilde{C}_G  \| (f,u_0) \|_Y 
\end{align}
with 
\[
\tilde{C}_G^2= \frac{2}{\underline{c}}
\max\Bigl\{ \frac{1}{2\cyoung}\Bigl(\myCC +1\Bigr), \, 
\frac{\myCC}{2}, \, \frac12 \Bigr\}.
\]
Here we have estimated 
$$\mathcal{E}[u](0)= \frac{1}{2} \| u_0 \|^2_{\Lp[n]{2}} 
+ \frac{\myCC}{2} \| \gradvec u_0 \|^2_{\Lp[n\times d]{2}} 
\leq \max\Bigl\{ \frac12 , \frac{\myCC}{2} \Bigr\}\|u_0\|^2_{(H^1)^n}.$$
3. By testing \eqref{GalerkinPDE} with $\dt u_m\in\galspa_m$ and integrating with respect to time, we can also estimate 
\begin{align}
	\begin{split}
\| \dt u_m \|_{L^2(T;(L^2)^n)} \leq \| \A u_m \|_{L^2(T;(L^2)^n)} &+ \| \B[u_m, u_*] \|_{L^2(T;(L^2)^n)}  \\ 
&+ \| \B[u_*,u_m] \|_{L^2(T;(L^2)^n)} + \| f \|_{L^2(T;(L^2)^n)}.
	\end{split}	
\end{align}
Applying the triangle inequality it is easy to see  that 
$$\| \A \|_{\linearspace(\tilde{X},L^2(T;(L^2)^n)} \leq \tilde{\eta} + 1  +  \max\{\| \lincofp \|_{L^\infty} , \| \lincofv \|_{L^\infty} \}$$ with $\tilde{\eta}=\max\{\etap,\etav \}$.
Using once more Lemma~\ref{NSFquadraticTermEstimateLemma} we altogether get
\begin{align}\label{estdtu}
\| \dt u_m \|_{L^2(T;(L^2)^n)} \leq \Bigl( \tilde{\eta} + 1  + \max\{\| \lincofp \|_{L^\infty} , \| \lincofv \|_{L^\infty} \} + 2 C_\B |\eps|\, \| u_* \|_{\tilde{X}}\Bigr)
\| u_m \|_{\tilde{X}}. 
\end{align}
Combining this with \eqref{enest_Galerkin_tilde} allows us to recover an estimate on the whole $X$ norm of $u_m$ and obtain \eqref{enest_Galerkin} with 
$$C_G=( \tilde{\eta} + 1  + \max\{\| \lincofp \|_{L^\infty} , \| \lincofv \|_{L^\infty} \} + 2 C_\B |\eps|\, \| u_* \|_{\tilde{X}})\tilde{C}_G+1, $$ provided $u_*\in $ $B^{\tilde{X}}_{r}, \gamma, \delta \in B^{L^\infty}_{\tilde{r}}$ with $r, \tilde{r}$ as in \eqref{smallness_ustar}, \eqref{gammdeltasmallness}.
Here we are using the norm defined by \eqref{normH1T}. 
Alternatively, we could combine \eqref{estdtu} with $\| u_m \|_{L^2(T;(L^2)^n)}\leq \| u_m \|_{\tilde{X}}$ to obtain an estimate on the standard ${H^1(T;(L^2)^n)}$ norm of $u_m$.
\end{proof}
%
%
%
%
\subsubsection{Existence and uniqueness}
In the last step of the proof we use weak convergence to establish convergence to the solution of the PDE with initial and boundary conditions. 
Due to the energy estimates we know that $u_m$ is contained in the ball with radius $R:=C_G \|(f,u_0)\|_{Y}$  in $X$, which is the dual of a separable Banach space. Thus, using the Banach-Alaoglu Theorem we get existence of a weak* accumulation point $u\in B_R^X$ of the sequence $u_m$. For simplicity of notation we denote the subsequence converging weakly* to $u$ again by $u_m$, that is  
$u_m\toweaklystar u$ in $X$, which implies
\begin{align} \label{WeakConvergenceDerivitaves}
	\begin{split}
		u_m \toweakly u, \quad 
		\dt u_m \toweakly \dt u, \quad 
		\partial_{ x_i } u_m \toweakly \partial_{x_i} u, \quad
		\lapl u_m \toweakly \lapl u \text{ in } L^2(T;(L^2)^n)
	\end{split}
\end{align}
for $i \in \{ 1, \ldots, \dime \}$.
Here we have used the fact that in a Hilbert space, weak* and weak convergence are equivalent.
Additionally, we make use of compactness of the embedding $X\embedded C(T;H^s)$, which allows us to add the strong limit
\begin{align} \label{StrongConvergence}
	\begin{split}
		u_m(0) \to u(0) \text{ in } (H^s)^n \text{ for any }s\in(0,1).
	\end{split}
\end{align}
For any fixed $\ell\in\N$ and any $w_\ell\in \galspa_\ell$, $\Phi_\ell\in L^2(T;\galspa_\ell)$, from the Galerkin equation \eqref{GalerkinPDE} we obtain for all $m\geq\ell$
\begin{equation}\label{Galerkinint}
\begin{aligned} 
	&\int_0^T\int_\Omega( \dt u_m(t) -\AAA u_m(t)+ \BBB[u_m(t),u_*(t)] + \BBB[u_*(t),u_m(t)] - f(t) )\Phi_\ell\, dx \, dt=0, \\
	&\int_\Omega (u_m(0) - 	u_0)w_\ell\, dx =0 .
\end{aligned}
\end{equation}
Note that due to Lemma~\ref{NSFquadraticTermEstimateLemma}, the linear operator $\B[\cdot,u_*]\in \linearspace(\tilde{X},L^2(T;(L^2)^n))$ is bounded, thus also weakly continuous; the same clearly holds for $\AAA$. 
Thanks to this and \eqref{WeakConvergenceDerivitaves}, \eqref{StrongConvergence}, we can take (weak*) limits in \eqref{Galerkinint} to obtain
\[
\begin{aligned} 
	&\int_0^T\int_\Omega( \dt u(t) -\AAA u(t)+ \BBB[u(t),u_*(t)] + \BBB[u_*(t),u(t)] - f(t) )\Phi_\ell\, dx \, dt=0 \\
	&\int_\Omega (u(0) - 	u_0)w_\ell\, dx =0. 
\end{aligned}
\]
Since $\ell$ was arbitrary, by density of $\bigcup_{\ell\in\N}L^2(T;\galspa_\ell)$ in $L^2(T,(L^2)^n)$ and of $\bigcup_{\ell\in\N}\galspa_\ell$ in $H^s$, this implies that $u$ satisfies the PDE and, furthermore, the initial conditions in an $L^2(T;(L^2)^n)$ and an $H^s$ sense, respectively.

Uniqueness follows directly from the a priori estimate \eqref{apriori} that by a testing strategy analogous to the one used in the proof of Lemma~\ref{lem:enest} actually holds not only for the limit of Galerkin solutions but for an arbitrary solution to \eqref{WeakSolution}.
Linearity of the problem implies that the difference $u_1-u_2$ between two potential solutions $u_1$, $u_2$ satisfies \eqref{WeakSolution} with $(f,u_0)=(0,0)$. Thus \eqref{apriori} yields that $\|u_1-u_2\|_{X}=0$.

\subsection{Local in time well-posedness by a Newton-Kantorovich type argument}
To prove existence and uniqueness of a solution to the nonlinear problem $F(u)=0$, that is of \eqref{2ndOrderNSF_ibvp} for given sufficiently small data $(f,u_0)=((\forcep,\forcev)^\transpose,(p_0,\velo_0)^\transpose)$ and $\gamma,\delta$, we apply the Newton-Kantorovich Theorem, whose proof in itself is basically a refined version of the Contraction Principle but additionally provides local quadratic convergence of Newton's method applied to the problem at hand. We recall it quoting from \cite{Ortega1968}, in a notation adapted to the one of this paper for the convenience of the reader.
\begin{theorem}\label{NewtonKantorowich}(Newton-Kantorovich Theorem, \cite{Kantorovich1948,KantorovichAkilov}) \\
Let $X$ and $Y$ be Banach spaces and $\F:\mathcal{D}(\F)\subseteq X \to Y$. Suppose that on an open convex set $D_*\subseteq \mathcal{D}(\F)$, $\F$ is Fr\'{e}chet differentiable and $\F'$ is Lipschitz continuous on $D_*$ with Lipschitz constant $K$
\[
\|\F'_{u_1}-\F'_{u_2}\|_{\linearspace(X,Y)}\leq K \|u_1-u_2\|_{X} , \quad u_1, \, u_2\in D_*.
\]
For some $u_*\in D_*$, assume that $\Gamma_*:={\F'_{u_*}}^{-1}$ is defined on all of $Y$ and that 
\begin{equation}\label{betaK}
\beta K \|\Gamma_*\F(u_*)\|_X\leq\frac12, \ \text{ where } \beta = \|\Gamma_*\|_{\linearspace(Y,X)}.
\end{equation}
Set 
\begin{equation}\label{rstart_rstarstar}
r^\pm=\frac{1}{\beta K}(1\pm\sqrt{1-2\beta K\|\Gamma_*\F(u_*)\|_X})
\end{equation}
and suppose that $B_{r^-}(u_*)\subseteq D_*$. 
Then the Newton iterates 
\begin{equation}\label{Newton}
u^{(k+1)}=u^{(k)}-{\F'_{u^{(k)}}}^{-1}\F(u^{(k)}), \, k=0,1,2,\ldots 
\end{equation}
starting at $u^{(0)}:=u_*$ 
are well-defined, lie in $B_{r^-}(u_*)$ and converge to a solution of $\F(u)=0$, which is unique in $B_{r^+}(u_*)\cap D_*$. Moreover, if strict inequality holds in \eqref{betaK}, the order of convergence is at least quadratic.
\end{theorem}

From Lemma~\ref{NSFDifferentaibleLemma} and Theorem~\ref{Theorem Linearized wellposed} we conclude that the assumptions of Theorem~\ref{NewtonKantorowich} are satisfied with $D_*:=U^{X}_{r}$ for $r$ defined in \eqref{smallness_ustar}, $\beta=C_G$, $K=C_\B|\eps|$ and so with $u_*:=0$, bearing in mind ${\F'_0}=\dt-\A$, the smallness requirement that we have to impose according to \eqref{betaK} is just 
\[
C_G C_\B |\eps| \, \|(\dt-\A)^{-1}(f,u_0)\|_X<\frac12.
\]
This is in fact a smallness condition on the data $(f,u_0)$ and $\gamma,\delta$ that can be satisfied by assuming
\begin{equation}\label{smallness_data}
C_G^2 C_\B |\eps| \, \|(f,u_0)\|_Y<\frac12.
\end{equation}
To make sure that $B_{r^-}^X(0)\subseteq D_*$, we additionally assume that 
\begin{equation}\label{smallness_data_1}
r^-<r \text{ for }
r^\pm:=\frac{1}{C_G C_\B |\eps|}(1\pm\sqrt{1-2C_G C_\B |\eps| \, \|{\F'_0}^{-1}(f,u_0)\|_X}).
\end{equation}
This yields the following well-posedness result.
\begin{theorem}\label{thm:wellposed_local}
For any $T>0$ and any small enough data $(f,u_0)\in Y$ such that \eqref{smallness_data}, \eqref{smallness_data_1} holds for $\tilde{r}, r$ defined in \eqref{gammdeltasmallness},\eqref{smallness_ustar}, there exists a solution $u$ to $\F(u)=0$.

That is, for small enough data $\lincofp, \lincofv \in L^\infty$, $\forcep\in L^2(T;L^2))$, $\forcev\in L^2(T;(L^2)^d)$, $p_0\in H_0^1$, $v_0\in (H_0^1)^d$, the initial boundary value problem \eqref{2ndOrderNSF_ibvp} is solvable in $X$.

The solution lies in $B_{r^-}^X$ and is unique in $B_{r^+}^X$ with $r^\pm$ defined as in \eqref{smallness_data_1}; additionally, it satisfies the a priori estimate \eqref{apriori}.

Moreover, the Newton iterates defined by \eqref{Newton} starting at $u^{(0)}:=0$ quadratically converge to $u$ in $X$. 
\end{theorem}

\subsection{Global in time well-posedness and exponential decay}
In this section we first of all note that the results from the last section extend to $T= \infty$.
To this end, let us write 
\begin{equation}
X(t):= H^1(t; (L^2)^{n}) \cap L^2 (t; ( H^2 \cap H^1_0 )^n ) \cap L^\infty(t, (H^1)^n ), 
\quad Y(t)=L^2(t;(L^2)^n)\times H_0^1.
\end{equation}
Indeed, all steps of the proofs of Theorem~\ref{Theorem Linearized wellposed} carry over to this case: \\
Due to \cite{Filippov:1988} we have well-posedness of the Galerkin ODE in the interval $(0,\infty)$, since $\A$, $\B[\cdot,u_*]$, $\B[u_*,\cdot]$ are in $L^1_{loc}(0,\infty)$ if $u_* \in X(\infty)$ and $f \in L^2(\infty;(L^2)^n)$. 

Independence of $T$ of the bound derived in the energy estimates provides us with 
$\|u_m\|_{X(\infty)}\leq$ $C_G \| ( f, u_0 ) \|_{Y(\infty )}$. 
Since our weak convergence argument for the PDE did not rely on any compactness argument, it directly carries over to $T=\infty$. The only place where we used compactness was attainment of the initial data, where we can restrict attention to a finite time interval, say, $[0,1]$, to invoke compactness of the embedding $X(1)\to C(1;H^s)$, while using $\|u_m\vert_{[0,1]}\|_{X(1)}\leq \|u_m\|_{X(\infty)} \leq C_G \|(f, u_0)\|_{Y(\infty )}$.

The Newton-Kantorovich arguments leading to Theorem~\ref{thm:wellposed_local} work on abstract spaces and therefore don't even need to be translated.

\medskip

As expected from the strong damping (see also Section~\ref{sec:Semigroup} below), we can also show exponential decay of solutions provided the forcing decays exponentially as well, that is,
\begin{equation}\label{exp_f}
f^\omega:t\mapsto \e^{\omega t} f(t) \in L^2(\infty;(L^2)^n) \text{ for some }\omega>0.
\end{equation}
Indeed, defining 
\[
\begin{aligned}
&(p^\lambda(t), \fat{v}^\lambda(t))^\transpose=u^\lambda(t):= \e^{\lambda t} u(t)=\e^{\lambda t}(p(t),\fat{v}(t))^\transpose,
\quad  E^\lambda(t):=\e^{-\lambda t}, \\ 
&\mathcal{A}^\lambda:= -\etap\lapl-\lambda\ident, \quad
\tilde{\mathcal{A}}^\lambda:= -\etav\laplvec-\lambda\ident,
\end{aligned}
\]
it is readily checked that $u^\lambda$ satisfies 
\begin{align} \label{2ndOrderNSF_ibvp_lambda}
\begin{split}
	\dt p^\lambda - \mathcal{A}^\lambda p^\lambda + (1 + \lincofp ) \divv \fat{v}^\lambda + \eps_1 E^\lambda p^\lambda \divv \fat{v}^\lambda + \eps_2 E^\lambda\grad p^\lambda \cdot \fat{v}^\lambda = \forcep^\lambda \\
	\dt \fat{v}^\lambda  - \tilde{\mathcal{A}} \fat{v}^\lambda +  (1 + \lincofv)\grad p^\lambda - \frac{\eps_3}{2} E^\lambda\grad (({p^\lambda})^2) + \frac{\eps_4}{2} E^\lambda\grad ({\fat{v}^\lambda} \cdot {\fat{v}^\lambda})= \forcev^\lambda \\
	(p^\lambda(0),\fat{v}^\lambda(0)) = (p_0, \fat{v_0}) \\
	p^\lambda = 0, \quad \fat{v}^\lambda = 0 \text{ on } \partial \Omega.
\end{split}	
\end{align}
Since we have $|E^\lambda|\leq1$, all estimates remain valid with slightly modified constants, as long as $\lambda\in(0,\min\{\etap,\etav\}\lambda_{\min}(-\lapl))$. Here $\lambda_{\min}(-\lapl)$ stands for the smallest eigenvalue of the Dirichlet Laplacian.
Thus, analogously to \eqref{apriori} we obtain
	\begin{align}\label{apriori_lambda}
		\| u^\lambda \|_{X(\infty)} \leq C_G \|(f^\lambda,u_0)\|_{Y(\infty)},
	\end{align}
where the right hand side is finite if $\lambda\leq\omega$.

\begin{theorem}\label{thm:wellposed_global_exp}
For any small enough data $\lincofp, \lincofv \in L^\infty$, $\forcep\in L^2(\infty;L^2))$, $\forcev\in L^2(\infty;(L^2)^d)$, $p_0\in H_0^1$, $v_0\in (H_0^1)^d$, the initial boundary value problem \eqref{2ndOrderNSF_ibvp} is solvable in $X(\infty)$.

The solution lies in $B_{r^-}^X$ and is unique in $B_{r^+}^X$ with $r^\pm$ defined as in \eqref{smallness_data_1}; additionally, it satisfies the a priori estimate \eqref{apriori}.

Moreover, the Newton iterates defined by \eqref{Newton} starting at $u^{(0)}:=0$ quadratically converge to $u$ in $X(\infty)$.

If additionally, for some $\lambda\in(0,\min\{\etap,\etav\}\lambda_{\min}( - \lapl))$, $\|f^\lambda\|_{L^2(\infty;(L^2)^n)}$ is finite and small enough, then $u^\lambda\in X(\infty)$ and in particular the decay estimate
\[
\|u(t)\|_{H_0^1}\leq C_0 \e^{-\lambda t} \|(f^\lambda, u_0)\|_{Y(\infty)}
\]
holds for some constant $C_0>0$ depending only on $\Omega$, $\lincofp$, $\lincofv$, $\lambda$, $\etap$, $\etav$, $\eps_i$ with $i \in \{1, \ldots, 4\}.$
\end{theorem}

%% file: Semigroup.tex
\section{Semigroup Theory}\label{sec:Semigroup}
\newcommand{\bfv}{\velo}
\newcommand{\bff}{\fat{f}}
\newcommand{\cA}{\mathfrak{A}}
\newcommand{\cB}{\mathfrak{B}}
\newcommand{\cC}{\mathfrak{C}}
\newcommand{\cD}{\mathfrak{D}}
\newcommand{\cH}{\mathfrak{H}}
\newcommand{\cM}{\mathfrak{M}}
\newcommand{\cAt}{\tilde{\cA}}
\newcommand{\cMt}{\tilde{\cM}}
\newcommand{\cAtt}{\tilde{\tilde{\cA}}}
\newcommand{\cBtt}{\tilde{\tilde{\cB}}}
In the small data and close to constant coefficient regime considered here, the expected qualitative behavior should already be visible in the linear constant coefficient setting by more straightforward means.
In this section, we thus derive some results from semigroup theory for the model equation \eqref{2ndOrderNSF} in the case $\eps=0$. The resulting linear equation differs from the linearization made in the Newton approach above but still gives some insight or actually some confirmation of the observations made for the more complicated system. In the following, we therefore analyze
\[
\begin{aligned}
&\partial_t p -\etap \lapl p + \nabla\cdot \bfv = \bff \\
&\partial_t \bfv -\etav \laplvec \bfv + \nabla p = g
\end{aligned}
\]
or equivalently (identifying $\AAA$ with its pointwise in time action) 
\[
u'(t)=\AAA u(t)
\]
with 
\begin{equation}
\AAA:= \left(\begin{array}{cc} 
-\cA&-\nabla\cdot\\
-\nabla & -\cAt
\end{array}\right),
\quad \cA:=-\etap \lapl, \quad \cAt:=-\etav \laplvec ,
\end{equation}
where $-\lapl$ and $- \laplvec$ are equipped with homogeneous Dirichlet boundary conditions. 
 
The operator $\AAA$ can naturally be decomposed into a skew adjoint and a self-adjoint positive definite part
\[
\AAA=\AAA_{skew}-\AAA_{pos} \text{ with }
\AAA_{skew}:= \left(\begin{array}{cc} 
0&-\nabla\cdot\\
-\nabla & 0
\end{array}\right),\
\AAA_{pos}:= \left(\begin{array}{cc} 
-\etap \lapl&0\\
0& -\etav \laplvec
\end{array}\right)
\]
on 
\[
H:=\cH^{1+d}, \quad \mathcal{D}(\AAA) := \mathcal{D}(\cA)^{1+d}, \quad
\cH:=L^2(\Omega), \quad  \mathcal{D}(\cA):= H^2(\Omega)\cap H_0^1(\Omega),
\]
where
\[
\AAA_{skew}^*=-\AAA_{skew}, \quad \AAA_{pos}^*=\AAA_{pos}, \quad (\AAA_{pos}u,u)\geq\min\{\etap,\etav\}\lambda_{\min}(-\lapl) \|u\|_H^2, \ u\in \mathcal{D}(\AAA).
\] 

Thus, in the undamped setting $\etap=\etav=0$, $\AAA$ generates a strongly continuous unitary group according to Stone's Theorem \cite[Theorem 3.24]{EngelNagel:2000}, whereas in the case of  strictly positive viscosity coefficients $\etap,\,\etav>0$ that we focus on in the sequel, it is dissipative, that is, 
\[
\text{ for all }u\in \mathcal{D}(\AAA): \ (\AAA u,u)\geq \tilde{\eta} \|u\|_H^2  
\]
with $\tilde{\eta}:=\min\{\etap,\etav\}\lambda_{\min}(-\lapl)>0$.
Indeed, $\AAA$ is invertible with bounded inverse
\[
\AAA^{-1}= -(-\AAA)^{-1}=-\left(\begin{array}{cc} 
\cA^{-1}+\cA^{-1}\nabla\cdot\cMt^{-1}\nabla\cA^{-1}&-\cA^{-1}\nabla\cdot\cMt^{-1}\\
-\cMt^{-1}\nabla \cA^{-1} & \cMt^{-1}
\end{array}\right)\in \linearspace(H, \mathcal{D}(\AAA)),
\]
where $\cMt = \cAt-\nabla\cA^{-1}\nabla\cdot$ and bounded invertibility $\cMt^{-1}\in \linearspace(\cH^d,\mathcal{D}(\cA)^d)$ follows from the fact that the operators $\cAt$ and $-\nabla\cA^{-1}\nabla\cdot\in L(\mathcal{D}(\cA)^d,\cH^d)$ are self-adjoint as well as positive and nonnegative, respectively. 
Moreover, obviously $\AAA$ is densely defined and it is readily checked (e.g., by the above stated boundedness of its inverse) that $\AAA$ is closed, that is,
\[
\text{ for all }(u_j)_{j\in\mathbb{N}}\in \mathcal{D}(\AAA): \ \Bigl(u_j\to u \text{ and }\AAA u_j \to f \text{ in }H\Bigr)
\ \Rightarrow \ \Bigl(u\in\mathcal{D}(\AAA) \text{ and }\AAA u = f \Bigr).
\]

As a consequence of the Lumer-Philips Theorem \cite[Theorem 3.15]{EngelNagel:2000} we have the following.
\begin{lemma}\label{lem:contractionsemigroup}
$\AAA$ is the infinitesimal generator of a contraction semigroup $S(t)= \e^{t\AAA}$ on $H$.
\end{lemma}

We now investigate analyticity and exponential stability of $S$, where the latter means that there exist constants $C$, $\gamma>0$ such that 
\[
\|S(t)\|_{\linearspace(H,H)}\leq C \e^{-\gamma t}.
\] 
To this end, we take a closer look at the resolvent along the imaginary axis to first of all prove that the resolvent set $\rho$ satisfies
\begin{equation}\label{rhoIm}
\imath\mathbb{R}\subseteq\rho(\AAA).
\end{equation}
Under this condition, given Lemma~\ref{lem:contractionsemigroup}, analyticity of the semigroup $S(t)$ is equivalent to
\begin{equation}\label{limsup1}
\limsup_{|\lambda|\to\infty}\|\lambda(\imath\lambda \ident -\AAA)^{-1}\|_{\linearspace(H,H)}<\infty
\end{equation}
cf., e.g., \cite[Theorem 1.3.3]{LiuZheng:1999}, while a result by Pr\"uss \cite[Corollary 4]{Pruess:1984} states that given  Lemma~\ref{lem:contractionsemigroup}, exponential stability of $S(t)$ is equivalent to \eqref{rhoIm} together with 
\begin{equation}\label{limsup0}
\limsup_{|\lambda|\to\infty}\|(\imath\lambda \ident - \AAA)^{-1}\|_{\linearspace(H,H)}<\infty.
\end{equation}

By splitting all involved quantities into their real and imaginary parts
\[
p=p_\Re+\imath p_\Im, \quad \bfv=\bfv_\Re+\imath \bfv_\Im, \quad 
g=g_\Re+\imath g_\Im, \quad \bff=\bff_\Re+\imath \bff_\Im,
\]
we obtain the following equivalence for the resolvent equation.
\[
 (\imath\lambda \ident -\AAA)\left(\begin{array}{c}p\\ \bfv\end{array}\right)=\left(\begin{array}{c}g\\ \bff\end{array}\right) \ \Leftrightarrow \
\left(\begin{array}{cccc} 
\cA& -\lambda &\nabla\cdot&0\\
\lambda&\cA & 0 &\nabla\cdot\\
\nabla& 0 &\cAt&-\lambda\\
0&\nabla&\lambda&\cAt
\end{array}\right)
\left(\begin{array}{c}p_\Re\\ p_\Im\\ \bfv_\Re\\ \bfv_\Im\end{array}\right)
=\left(\begin{array}{c}g_\Re\\ g_\Im\\ \bff_\Re\\ \bff_\Im\end{array}\right)
\]
Some elementary but tedious computations yield the following expression for the operator matrix appearing here 
\begin{equation}\label{resolventIm}
\begin{aligned}
&\left(\begin{array}{cccc} 
\cA& -\lambda &\nabla\cdot&0\\
\lambda&\cA & 0 &\nabla\cdot\\
\nabla& 0 &\cAt&-\lambda\\
0&\nabla&\lambda&\cAt
\end{array}\right)^{-1}
=\left(\begin{array}{cc} 
D-DBMB^*D& -DBM\\
MB^*D&M\end{array}\right)
\end{aligned}
\end{equation}
with
\begin{equation}
\begin{aligned}
&D=\left(\begin{array}{cc} 
\cA& -\lambda \\
\lambda&\cA \end{array}\right)^{-1}
=\left(\begin{array}{cc} 
\cB^{-1}&\lambda \cB^{-1}\cA^{-1}\\
-\lambda \cB^{-1}\cA^{-1}&\cB^{-1}\end{array}\right),
\quad D^*=D^T, \\
&B=\left(\begin{array}{cc} 
\nabla\cdot&0\\
0&\nabla\cdot\end{array}\right),\quad
B^*=-\left(\begin{array}{cc} 
\nabla&0\\
0&\nabla\end{array}\right),\\
&M=\left(\left(\begin{array}{cc} 
\cAt& -\lambda \\
\lambda&\cAt \end{array}\right)
+B^*DB\right)^{-1}
=\left(\begin{array}{cc} 
\cAtt^{-1}-\cAtt^{-1}\cC\cBtt^{-1}\cC\cAtt^{-1}&\cAtt^{-1}\cC\cBtt^{-1}\\
-\cBtt^{-1}\cC\cAtt^{-1}&\cBtt^{-1}\end{array}\right),\\
&\cB=\cA+\lambda^2\cA^{-1}, \quad
\cC=\lambda(\ident+\nabla\cB^{-1}\cA^{-1}\nabla\cdot), \\
&\cAtt=\cAt-\nabla\cB^{-1}\nabla\cdot, \quad
\cBtt=\cAtt+\cC\cAtt^{-1}\cC
\end{aligned}
\end{equation}
from which we can conclude \eqref{rhoIm}.
The following mapping properties as well as large $\lambda$ asymptotics hold.
\begin{equation}\label{lambda2BinvAinv}
\begin{aligned}
&\|\lambda^2\cB^{-1}\cA^{-1}\|_{\linearspace(\cH,\cH)}
=\sup_{p\in \mathcal{D}(\cA)\setminus\{0\}}\frac{(\lambda^2\cB^{-1}\cA^{-1}p,p)}{\|p\|_{\cH}^2}
=\sup_{q\in\mathcal{D}(\cA^2)\setminus\{0\}}\frac{\|q\|_{\cH}^2}{(\lambda^{-2}\cB\cA q,q)}\\
&=\sup_{q\in\mathcal{D}(\cA^2)\setminus\{0\}}\frac{\|q\|_{\cH}^2}{(\ident+\lambda^{-2}\cA^2 q,q)}\leq 1
\end{aligned}
\end{equation}
by density of $\mathcal{D}(\cA)$ in $\cH$ and with $q=\lambda(\cB^{-1}\cA^{-1})^{1/2}p$.
On the other hand, similarly we have
\begin{equation}\label{BinvA}
\|\cB^{-1}\cA\|_{\linearspace(\cH,\cH)}
=\|(\ident+\lambda^2\cA^{-2})^{-1}\|_{\linearspace(\cH,\cH)}\leq 1.
\end{equation}
Thus by interpolation, using the fact that $\cB$ is self-adjoint and commutes with $\cA$, we conclude 
\[
\|\lambda\cB^{-1}\|_{\linearspace(\cH,\cH)} 
=\|\lambda^2\cB^{-2}\|_{\linearspace(\cH,\cH)}^{1/2}
=\|\lambda^2\cB^{-1}\cA^{-1}\cB^{-1}\cA\|_{\linearspace(\cH,\cH)}^{1/2}
\leq 1.
\]
Altogether, this implies uniform boundedness of $\lambda D$ and analogously\footnote{Even without using the representation of the inverse but just the fact that $M^{-1}\succeq \left(\begin{array}{cc} 
\cAt& -\lambda \\
\lambda&\cAt \end{array}\right)$ suffices.}, of $\lambda M$ with
\[
\|\lambda D\|_{\linearspace(\cH^2,\cH^2)}\leq 2, \quad \|\lambda M\|_{\linearspace(\cH^{2d},\cH^{2d})}\leq 2.
\]
Additionally, for the off-diagonal entries in \eqref{resolventIm} with 
$q:=\left(\begin{array}{c}q_\Re\\q_\Im\end{array}\right)$, 
$\tilde{q}:=\text{diag}(\cA^{-1/2},\cA^{-1/2})q$
we have 
\[
\begin{aligned}
\|B^*D q\|_{L^2(\Omega)^{2d}}^2 &= \|D q\|_{H_0^1(\Omega)^2}^2 
= \tfrac{1}{\etap} \|\text{diag}(\cA^{1/2},\cA^{1/2})D q\|_{L^2(\Omega)^2}^2
= \tfrac{1}{\etap} (q,D^T\text{diag}(\cA,\cA)D q)_{L^2(\Omega)^2}\\
&= \tfrac{1}{\etap} (\tilde{q}, \text{diag}(\cD,\cD)\tilde{q})_{L^2(\Omega)}
\leq \tfrac{2}{\etap} \|\tilde{q}\|_{L^2(\Omega)^2}^2
=2\|q\|_{H^{-1}(\Omega)^2}^2
\leq 2 C_{PF}^2 \|q\|_{L^2(\Omega)^2}^2, 
\end{aligned}
\]
\text{where }$\cD=(\cB^{-1}\cA)^2+\lambda^2\cB^{-2}$ and $\|\cD\|_{\linearspace(\cH^2,\cH^2)}\leq2$.
Here $C_{PF}$ is the constant in the Poincar\'{e}-Friedrichs inequality
\[
\|v\|_{L^2(\Omega)}\leq C_{PF} \|\nabla v\|_{L^2(\Omega)}, \quad v\in H_0^1(\Omega)
\]
and $H^{-1}$ denotes the dual space of $H^1_0$.
Thus 
\[
\|B^*D\|_{\linearspace(\cH^2,\cH^{2d})}\leq \sqrt{2}C_{PF}, \quad 
\|B^*D^T\| _{\linearspace(\cH^2,\cH^{2d})}\leq \sqrt{2}C_{PF}, \quad 
\|DB\| _{\linearspace(\cH^{2d},\cH^2)}\leq \sqrt{2}C_{PF},
\]
where the second estimate follows analogously to the first one and the last one from the second one by duality. 
Note that the operators $B^*D$, $DB$ are even compact.

Altogether we therefore obtain \eqref{limsup1}.

\begin{lemma}\label{lem:resolvent}
The semigroup $S(t)$ generated by $\AAA$ is analytic and exponentially stable. 
\end{lemma}